\documentclass[10pt,reqno,fleqn]{amsart}
\usepackage[utf8]{inputenc}
\usepackage{amsmath}
\usepackage{setspace}
\usepackage{bbm}
\usepackage{amsthm}
\usepackage{acronym}
\usepackage{amscd}
\usepackage{amssymb}
\usepackage{mathtools}
\usepackage{stmaryrd}
\usepackage{pifont}
\newcommand\MTkillspecial[1]{
\bgroup
\catcode`\&=9
\let\\\relax%
\scantokens{#1}%
\egroup
}
\DeclarePairedDelimiter\abs\lvert\rvert
\reDeclarePairedDelimiterInnerWrapper\abs{star}{
\mathopen{#1\vphantom{\MTkillspecial{#2}}\kern-\nulldelimiterspace\right.}
#2
\mathclose{\left.\kern-\nulldelimiterspace\vphantom{\MTkillspecial{#2}}#3}}

\usepackage{epsfig}
\usepackage{verbatim}
\usepackage{graphicx}
\usepackage{amsthm}
\usepackage{breqn}
\usepackage{amsaddr}
\usepackage{color}
\usepackage{url}
\usepackage[bookmarks=true]{hyperref}
\usepackage{bookmark}
\newtheorem*{prop*}{Proposition}
\newtheorem*{theorem*}{Theorem}

\usepackage{tikz-cd}

\title{Weyl's law for arbitrary Archimedean type}
\author{Ayan Maiti$^1$}
\address{$^1$Department of Mathematics, Purdue University}
\email{$^1$maitia@purdue.edu}

\newtheorem{lemma}{Lemma}

\begin{document}

\maketitle

\begin{abstract}
We generalize the work of Lindenstrauss and Venkatesh establishing Weyl's Law for cusp forms from the spherical spectrum to arbitrary Archimedean type. Weyl's law for the spherical spectrum gives an asymptotic formula for the number of cusp forms that are bi-$K_{\infty}$ invariant in terms of eigenvalue $T$ of the Laplacian. We prove an analogous asymptotic holds for cusp forms with Archimedean type $\tau$, where the main term is multiplied by $\dim{\tau}$. While in the spherical case the surjectivity of the Satake Map was used, in the more general case that is not available and we use Arthur's Paley-Wiener theorem and multipliers.

\end{abstract}
\subsection*{Mathematics Subject Classification} Primary $11F72$, $22E55$, Secondary $32N15$.
\section{Introduction}
The purpose of this article is to prove Weyl's law for the cuspidal automorphic forms, generalizing a result of Lindenstrauss and Venkatesh \cite{LV} from spherical case (bi-$K_{\infty}$ invariant) to the cuspforms of arbitrary $K_{\infty}$-type.\\
Let $G$ be a semisimple linear algebraic group which is split and adjoint over $\mathbb{Q}$. Let $G(\mathbb{R})$ be the $\mathbb{R}$-points of $G$. Let $\Gamma \subset G(\mathbb{R})$ be an arithmetic subgroup, which we assume to be torsion-free for simplicity. Let $K_{\infty}$ be a maximal compact subgroup of $G(\mathbb{R})$. Let $L^2(\Gamma \backslash G(\mathbb{R}))$ be the space of square integrable $\Gamma$ invariant functions on $G(\mathbb{R})$. Let $\mathcal{Z}(\mathfrak{g}_{\mathbb{C}})$ be the center of universal enveloping algebra of the complexification of the Lie algebra $\mathfrak{g}$ of $G(\mathbb{R})$. A  cuspform for $\Gamma$ is a smooth and $K_{\infty}$-finite complex-valued functions $f$, which is a simultaneous eigenfunction of $\mathcal{Z}(\mathfrak{g}_{\mathbb{C}})$ and which satisfies
$$\int\limits_{\Gamma \bigcap N_{P}(\mathbb{R})\backslash N_{P}(\mathbb{R})}f(nx)dn=0,$$
for all unipotent radicals $N_{P}$ of proper rational parabolic subgroups $P$ of $G$ \cite{La}. It can be shown that cusp forms are square- integrable. Let $L_{\text{cusp}}^2(\Gamma \backslash G(\mathbb{R}))$ be the closure of the linear span of all cusp forms. Let $R$ be the right regular representation of $G(\mathbb{R})$ on $L^2(\Gamma \backslash G(\mathbb{R}))$. Suppose $(\tau,V_{\tau})$ denotes an irreducible finite dimensional representation of the maximal compact $K_{\infty}$. We let
$$(L^2(\Gamma \backslash G(\mathbb{R}))\otimes V_{\tau})^{K_{\infty}}$$
be the space of homogeneous vector bundle on the Riemannian symmetric space $G(\mathbb{R})/K_{\infty}$. These are space of functions that satisfies the following condition:
$$f(gk)=\tau(k^{-1})f(g).$$
Let $\Omega_{G(\mathbb{R})}$ be the casimir operator in $\mathcal{Z}(\mathfrak{g}_{\mathbb{C}})$, the center of the universal enveloping algebra. Then $-\Omega_{G(\mathbb{R})} \otimes Id$ induces a self adjoint operator $\Delta_{\tau}$ whose restriction to
$$L_{\text{cusp}}^2(\Gamma \backslash G(\mathbb{R}),\tau):= (L_{\text{cusp}}^2(\Gamma \backslash G(\mathbb{R}))\otimes V_{\tau})^{K_{\infty}}$$
has pure point spectrum with finite multiplicities. Let us denote them as $$0\leq \nu_{1}(\tau) < \nu_{2}(\tau)\cdots$$ with finite multiplicities. Suppose $\mathcal{E}(\nu_{i}(\tau))$ denotes the respective eigenspace corresponding to the eigenvalue $\nu_{i}$. We define the eigenvalue counting function as
$$N_{\text{cusp}}^{\Gamma}(\nu,\tau)= \sum_{\nu_{i}(\tau)\leq \nu}\text{dim}(\mathcal{E}(\nu_{i}(\tau))).$$
\par Let $\mathbb{H}$ be the upper half plane and let $\Gamma$ be a congruence subgroup of $SL(2,\mathbb{Z})$. Let $\Delta$ be the hyperbolic laplacian on $\mathbb{H}$. Using the notations above let $N_{\text{cusp}}^{\Gamma}(\nu)$ be the eigenfunction counting function for eigenvalues upto $\nu$. Selberg \cite{Se}, using his celebrated trace formula for the group $SL(2,\mathbb{R})$, proved the following version of Weyl's law:
$$\lim_{\nu \to \infty}N_{\text{cusp}}^{\Gamma}(\nu)\sim \text{Vol}(\Gamma \backslash \mathbb{H})\frac{\nu}{4\pi}.$$
If $d=\text{dim}G(\mathbb{R})/K_{\infty} $, then it has been conjectured by Sarnak \cite{Sa} that for $d>1$ and for an irreducible lattice $\Gamma$:
$$\limsup_{\nu \to \infty}\frac{N_{\text{cusp}}^{\Gamma}(\nu)}{\nu^{d/2}}\sim \frac{\text{vol}(\Gamma \backslash G)}{(4\pi)^{d/2}\Gamma(d/2+1)},$$
where $\Gamma(n)$ denotes the Gamma function.\\
A similar conjecture was made by M\"uller \cite{Mu1} for $N_{\text{disc}}^{\Gamma}(\nu,\tau)$, the counting function for the discrete spectrum of the Laplace operator $\Delta_{\tau}$. This conjecture states that for any arithmetic subgroup and any $K_{\infty}$-type $\tau$ we have:
$$\limsup_{\nu \to \infty}\frac{N_{\text{cusp}}^{\Gamma}(\nu,\tau)}{\nu^{d/2}}\sim \frac{\text{vol}(\Gamma \backslash G)\text{dim}(\tau)}{(4\pi)^{d/2}\Gamma(d/2+1)}.$$
Up to now this conjecture has been proved for the following cases: for the congruence subgroups of $SO(n,1)$ by Reznikov \cite{Re}, congruence subgroups of $\text{Res}_{F/\mathbb{Q}}SL_{2}$, where $F$ is totally real field, by Efrat \cite{Ef}, for $\Gamma=SL_{3}(\mathbb{Z})$ by Stephen D. Miller \cite{Mi}, and for torsion free arithmetic subgroups of $SL_{n}(\mathbb{R})$ by M\"uller \cite{Mu}.\\
Labesse and M\"uller \cite{LM} proved a weak version of Weyl's law for almost simply connected, simply connected, semisimple algebraic groups. To explain their method we introduce the following notations:
Let $G(\mathbb{A})$ be the group of adelic points of algebraic group $G$ defined over $\mathbb{Q}$. Let $K=K_{\infty}\times K_{f}$ be an open compact subgroup of $G(\mathbb{A})$. Then by Strong Approximation Theorem we have that for $\Gamma=G(\mathbb{Q}) \cap K_{f}$:
$$L_{\text{cusp}}^2(G(\mathbb{Q}) \backslash G(\mathbb{A})/K_{f}) \simeq L_{{\text{cusp}}}^2(\Gamma \backslash G(\mathbb{R})).$$
To understand the spectral side of the Arthur-Selberg trace formula we need to give a representation theoretic point of view of the eigenvalue counting function $N_{\text{cusp}}^{\Gamma}(\nu,\tau)$. Let $\Pi_{\text{cusp}}(G(\mathbb{A}))$ be the unitary irreducible cuspidal subrepresentations of the regular representation of $G(\mathbb{A})$ on $L_{\text{cusp}}^2(G(\mathbb{Q}) \backslash G(\mathbb{A})/K_{f})$. Let $\Pi_{\text{cusp}}(G(\mathbb{R}))$ be the subrepresentations of the regular representations of $G(\mathbb{R})$ acting on 
$$L_{\text{cusp}}^2(G(\mathbb{Q}) \backslash G(\mathbb{A})/K_{f}).$$
Any element $\pi \in \Pi_{\text{cusp}}(G(\mathbb{A}))$ can be written as $\pi=\pi_{\infty}\otimes \pi_{f}$, where $\pi_{\infty} \in \Pi_{\text{cusp}}(G(\mathbb{R}))$. Let $H_{\pi_{\infty}}(\tau)$ be the $\tau$-isotypical subspace of $(\pi_{\infty},H_{\pi_{\infty}})$. Let $H_{\pi_{f}}^{K_{f}}$ be the subspace of $K_{f}$-fixed vectors in $(\pi_{f},H_{\pi_{f}})$. Let $m(\pi_{\infty})$, resp $m(\pi)$ denote the multiplicity with which $\pi_{\infty}$, resp $\pi$ occurs as a subrepresentation of $G(\mathbb{R})$, resp $G(\mathbb{A})$ in the discrete subspace $L_{\text{cusp}}^2(G(\mathbb{Q}) \backslash G(\mathbb{A})/K_{f})$. Then we have the following :
$$m(\pi_{\infty})=\sum_{\pi^{'}\in \Pi_{\text{cusp}}(G(\mathbb{A}))}m(\pi^{'})\text{dim}H_{\pi_{f}}^{K_{f}}$$
for all $\pi^{'}$ such that $\pi_{\infty}^{'}=\pi_{\infty}$. Suppose $\nu_{\pi}$ denotes the Casimir eigenvalue of $\pi_{\infty}$. Then we take the sub-collection $\Pi_{\text{cusp}}(G(\mathbb{A}))_{\nu}$ such that $|\nu_{\pi}|\leq \nu$. Similarly we define $\Pi_{\text{cusp}}(G(\mathbb{R}))_{\nu}$. Then we have :
$$\sum_{\pi_{\infty}\in \Pi_{\text{cusp}}(G(\mathbb{R}))_{\nu}}m(\pi_{\infty})\text{dim}(\text{Hom}_{K_{\infty}}(H_{\pi_{\infty}}(\tau),V_{\tau}))= N_{\text{cusp}}^{\Gamma}(\nu,\tau).$$
The usual idea of proving the asymptotic formula for the counting functions is to apply the Arthur-Selberg trace formula for a family of test functions on $G(\mathbb{A})$ whose Archimedean part arise from the integral kernel function of the integral operator $e^{-t\Delta_{\tau}}$, for $0\leq t <1$, and the non-Archimedean parts are idempotents $e_{K_{f}}$. In the spectral side the terms corresponding to the innerproduct of Eisenstein series will contribute trivially when $t \to 0$ as shown by M\"ueller \cite{Mu} in the case of $SL_{n}(\mathbb{R})$. But the calculation is delicate for arbitrary groups. In this regard it will be useful to find test functions such that convolution operators with respect to them have purely cuspidal image.\\ 
Let $S$ be a finite set of non-archimedean places. Then, following the simple trace formula introduced by Flicker and Kazdhan \cite{FK}, Labesse and M\"uller \cite{LM} considered the test functions decomposed as $f=f_{\infty}\otimes f_{S}\otimes e_{K_{f}^S}$, where $f_{S}$ are the pseudo-coefficients of Steinberg representation of $G(\mathbb{Q}_{S})$ acting on $L_{\text{cusp}}^2(G(\mathbb{Q}) \backslash G(\mathbb{A})/K_{f})$. Hence the image of the right regular representation with respect to the above test function projects into the subspace $L_{\text{cusp}}^2(G(\mathbb{Q}) \backslash G(\mathbb{A})/K_{f},S)$, generated by the vectors of automorphic representations which are Steinberg at places in $S$. Define the eigenvalue counting function $N_{\text{cusp}}^{\Gamma}(\nu,\tau,S)$ with respect to $S$ in $L_{\text{cusp}}^2(G(\mathbb{Q}) \backslash G(\mathbb{A})/K_{f},S)$. Using this idea they were able to show that:
$$\limsup_{\nu \to \infty}\frac{N_{\text{cusp}}^{\Gamma}(\nu,\tau,S)}{\nu^{d/2}}= \frac{C_{S}(\Gamma)\text{vol}(\Gamma \backslash G)\text{dim}(\tau)}{(4\pi)^{d/2}\Gamma(d/2+1)}.$$
But the non-triviality of the constant $C_{S}(\Gamma)$ would depend on the choice of the compact set $K_{f}$, as $\Gamma = G(\mathbb{Q}) \cap K_{f}$, where $K_{p}$ for $p\in S$ lies inside the minimal parahoric compact subgroup.\\
To get the full Weyl's law for spherical cuspforms on semisimple Algebraic group of split and adjoint type, Lindenstrauss and Venkatesh \cite{LV} were able to find a collection of test functions from the spherical Hecke algebra $C_{c}^{\infty}(G_{S}//K_{S})$ that has purely cuspidal image, where $S$ is a set of places containing the Archimedean places, and $K_{S}=K_{\infty}\times K_{f}$, where $K_{\infty}$ is a maximal compact subgroup of $G(\mathbb{R})$ and $K_{f}$ is a hyperspecial maximal compact subgroup of $G_{S\backslash \infty}$. They use the Satake Isomorphism for spherical Hecke algebra to prove the existence of such functions. Also the arithmetic subgroup is chosen as a congruence subgroup of $G(\mathbb{Z}[S^{-1}])$, so that the projection of $\Gamma$ on the finite number of inequivalent conjugacy classes of parabolic subgroups of $G_{S}$ have large center. Hence there are constraints on the spectral parameters of Eisenstein series at different places. They used the Paley-Wiener Theorem for the spherical functions to choose the family of test functions of the form $\phi_{n}\star f_{t}$, for $n\in \mathbb{Z}, 0\leq t<1$, where $\phi_{n}$'s are the family of test functions constructed to get the purely cuspidal image. Now instead of using Arthur's trace formula they use a partial trace formula introduced by Miller \cite{Mi} to get the lower bound with the same constant as Donnelly's upperbound \cite{Do}. \\
Jack Buttcane in \cite{BU} proved the Weyl's Law for the case $G=GL(3,\mathbb{R})$, using Kuznetsov trace formula.\\

Our theorem, which is a generalization of the result of Lindenstrauss and Venkatesh \cite{LV}, to the case of cuspforms of arbitrary $K_{\infty}$-type is the following:
\begin{theorem*}
Let $G$ be a semi-simple, split, adjoint  linear algebraic group over $\mathbb{Q}$. Let $G_{\infty}=G(\mathbb{R})$ be the real points of $G$, and let $\Gamma$ be a torsion free arithmetic subgroup of $G$. Suppose $d=\text{dim}(G_{\infty}/K_{\infty})$. Let $(\tau,V_{\tau})$ be an irreducible finite-dimensional representation of $K_{\infty}$. Let $N_{\text{cusp}}^{\Gamma}(T,\tau)$ be the counting function of the cuspidal eigenfunctions of $\Delta_{\tau}$ with eigenvalue $\leq T$. Then we have the following asymptotic formula:
\begin{eqnarray}\label{1}
\frac{N_{\text{cusp}}^{\Gamma}(T,\tau)}{T^{d/2}}\sim \frac{\text{dim}(\tau)
\text{vol}(\Gamma \backslash G_{\infty})}{(4\pi)^{d/2}\Gamma(d/2+1)}, \quad \text{as} \quad T \to \infty.
\end{eqnarray}
\end{theorem*}
We follow the same methodology as in Lindenstrauss-Venkatesh \cite{LV} to prove the Weyl's law (without the remainder term) for an arbitrary irreducible $K_{\infty}$-type. But as the Abel-Satake map is not surjective in this case, we can not use the same construction to get the non-trivial test function whose image under convolution lies inside the cuspidal space. Hence, we use the Arthur's Paley-Wiener theorem for Archimedean \cite{Ar} and non-Archimedean \cite{Ar3} cases. We note that we are only concerned with the main term of the Weyl's law. To get an estimation of the error terms as derived by M\"uller, one has to use the full trace formula and estimates on the constant terms of the Eisenstein series.\\
This draft is organized as follows: in section 2 we discuss the necessary preliminaries of Harmonic Analysis of reductive groups over real and $p$-adic fields. In section 3 we prove some estimates of Plancherel Measure necessary to prove the asymptotic formula of the Main term of the Weyl's law, section 4 we prove a condition on test functions so that their convolution image is purely cuspidal, section 5 and 6 we provide the analysis to derive Theorem mentioned above.

I would like to thank my advisor Professor Mahdi Asgari for suggesting this problem, and Professors Roberto Camporesi, Jayce Getz, Werner M\"uller, Sug Woo Shin and Roger Zierau for their help and useful inputs at various stages of this work.

\section{Preliminaries}
In this section we will recall some basic facts of Harmonic Analysis.(see \cite{LV}).
\subsection{Parabolic Subgroups} Let $G$ be a semisimple split adjoint linear algebraic group over $\mathbb{Q}$. Let $S$ be a finite set of places containing $\infty$. We fix a minimal parabolic, i.e. a Borel subgroup, $P_{0} \supset A_{0}$, where $A_{0}$ is a maximal $\mathbb{Q}$-spilt torus. Suppose $N_{0}=R_{u}(P_{0})$ is the unipotent radical of $P_{0}$. We have the Levi decomposition $P_{0}=M_{0}N_{0}$. Let $P$ be a parabolic subgroup containing $P_{0}$ with a Levi decomposition $P=M_{P}N_{P}$. Moreover we let $A_{P}=$ Split part of $Z(M_{P})$, where $Z$ denotes the center.\\
Let $F=\mathbb{Q}_{p}$ or $\mathbb{R}$. 
Fix a maximal split torus $A_{0}$ in $G(F)$. We denote by
$W = W(G, A_{0})$ the Weyl group of $G(F)$ with respect to $A_{0}$. Let $\Phi = \Phi(G, A_{0})$ be the set of roots. Fix a minimal parabolic subgroup $B$ containing $A_{0}$. The choice of $B$ determines the set of simple roots $\Pi$ and the set of positive roots $\Phi^{+} \subset \phi$. If $\alpha \in \Phi^{+}$, we write $\alpha > 0$.\\
Let $P = MN \subset G(F)$ be a standard parabolic subgroup of $G(F)$. We denote
by $\Pi_{M} \subset \Phi$ the corresponding set of simple roots. Let $A_M$ be the split component of the center of $M$, $X(M)_F$ the group of $F$-rational characters of $M$. If $\Pi_{M} = \Theta$, we also use $A_{\Theta}$ to denote $A_M$ . Hence, $A_{\emptyset} = A$ and $A_{\Pi} = A_G$.\\
The restriction homomorphism $X(M)_F \mapsto X(A_M)_F$ is injective and has a finite cokernel. Therefore, we have a canonical linear isomorphism:
$$\mathfrak{a}_{M}^*=X(M)_{F}\otimes_{\mathbb{Z}}\mathbb{R}\cong X(A_{M})_{F}\otimes_{\mathbb{Z}}\mathbb{R}.$$
If $L$ is a standard parabolic subgroups such that $L \subset M$, then
$$A_{M}\subset A_{L}\subset L \subset M.$$
The restriction $X(M)_{F} \mapsto X(L)_{F}$ induces an injective map and its restriction induces a linear injection $i_{M}^L : \mathfrak{a}_{M}^* \mapsto \mathfrak{a}_{L}^*$. The restriction $X(A_{L})_{F} \mapsto X(A_{M})_{F}$ induces a linear surjection $r_{M}^{L}: \mathfrak{a}_{L}^* \mapsto \mathfrak{a}_{M}^*$. Let $(\mathfrak{a}_{M}^{L})^*$ be the kernel of the restriction $r_{M}^{L}$. Then
$$\mathfrak{a}_{L}^*=i_{M}^{L}(\mathfrak{a}_{M}^*)\oplus (\mathfrak{a}_{M}^{L})^*.$$
There is a homomorphism $H_{M}: M \mapsto \mathfrak{a}_{M}=Hom(X(M),\mathbb{R})$ such that:
\[
    \left|\nu(m)\right|_{F}= 
\begin{cases}
    q^{(\nu,H_{M}(m))},& \text{if } F=\mathbb{Q}_{p}\\
    e^{(\nu,H_{M}(m))},& \text{if } F=\mathbb{R}
\end{cases}
\]
, for all $m \in M$ and $\nu \in X(M)_{F}$.\\
We set $G_{S}=G(\mathbb{Q}_{S})$, $A_{0,S}=A_{0}(\mathbb{Q}_{S})$, $M_{0,S}=M_{0}(\mathbb{Q}_{S})$ and $N_{0,S}=N_{0}(\mathbb{Q}_{S})$. We denote a parabolic subgroup over $\mathbb{Q}_{S}$ as $P_{S}=M(\mathbb{Q}_{S})N(\mathbb{Q}_{S})$ with its corresponding Levi decomposition. We can think of this parabolic as a direct product of parabolic subgroups of a product of groups. Let $G_{\infty}=G(\mathbb{R})$. We have an Iwasawa decomposition $G_{\infty}=N_{\infty}A_{\infty}^{o}K_{\infty}$, where $K_{\infty}$ is a maximal compact subgroup of $G(\mathbb{R})$.\\
Let $K_{S}=K_{\infty}\prod\limits_{p<\infty}G(\mathbb{Z}_{p})$, where $G(\mathbb{Z}_{p})$ is a maximal compact subgroup of $G(\mathbb{Q}_{p})$ for all prime $p$. We will assume that $S$ has the property that, for each finite $p \in S$ and for each parabolic $P(\mathbb{Q}_{p})$ containing $A_{0}(\mathbb{Q}_{p})$, $K_{p} \bigcap M_{P} (\mathbb{Q}_{p})$ is the stabilizer in $M_{P}(\mathbb{Q}_{p})$ of a special vertex in the building of $M_{P}(\mathbb{Q}_{p})$; and moreover this vertex belongs to the apartment associated to the maximal torus $A_{0}(\mathbb{Q}_{p})$.
This condition is satisfied for almost all finite $p$. Moreover, $K_{\infty} \bigcap M_{P}(\mathbb{R})$ is a maximal compact subgroup of  $M_{P}(\mathbb{R})$, and $K_{S} \bigcap M(\mathbb{Q}_{S})$ is a maximal compact subgroup of $M(\mathbb{Q}_{S})$.\\
The map $N(\mathbb{Q}_{S})\times M(\mathbb{Q}_{S})\times K_{S} \mapsto G_{S}$ is surjective (Iwasawa). We equip each $G(\mathbb{Q}_{p})$, for $p$ finite, with the Haar measure which assigns $G(\mathbb{Z}_{p})$ the mass 1. We equip $K_{\infty}$ with the Haar measure of mass 1, and then choose the Haar measure on $G_{\infty}$ which is compatible with the Riemannian metric defined on the Riemannian symmetric space $G_{\infty}/K_{\infty}$.
Let $\Phi^{+}$ be the system of positive roots of $A_{0,S}$ with respect to $N_{0,S}$ and let $\Delta \subset \Phi^{+}$ be the set of simple roots. Let $\delta_{S}$ be the square root of the modulus character of $A_{0,S}$.\\
The following lemma which is due to Harish-chandra going to describe the correspondence between parabolic subgroups of $G(\mathbb{Q}_{p})$ contained in some parabolic subgroup $Q(\mathbb{Q}_{p})$ and the parabolic subgroups of $M_{Q}(\mathbb{Q}_{p})$ for all $p \in S$.
\begin{lemma}
There is an one to one correspondence between Parabolic subgroup $P(\mathbb{Q}_{p})$ of $G(\mathbb{Q}_{p})$ which are contained in $Q(\mathbb{Q}_{p})$, and parabolic subgroups $^{*}P(\mathbb{Q}_{p})$ of $M_{Q}(\mathbb{Q}_{p})$. The correspondence are as follows:
If $Q(\mathbb{Q}_{p})=M_{Q}(\mathbb{Q}_{p})N_{Q}(\mathbb{Q}_{p})$ and  $P(\mathbb{Q}_{p})=M_{P}(\mathbb{Q}_{p})N_{P}(\mathbb{Q}_{p})$ are the corresponding Levi decompositions, then the Levi decomposition of $^{*}P_{Q}=P(\mathbb{Q}_{p})\cap M_{Q}(\mathbb{Q}_{p})=M_{P}(\mathbb{Q}_{p})N_{Q}^{P}(\mathbb{Q}_{p})$, where $A_{P}(\mathbb{Q}_{p})=A_{Q}^{P}(\mathbb{Q}_{p})A_{Q}(\mathbb{Q}_{p})$, $N_{P}(\mathbb{Q}_{p})=N_{Q}^{P}(\mathbb{Q}_{p})N_{Q}(\mathbb{Q}_{p})$.
\end{lemma}
\subsection{Congruence Subgroup} We choose a congruence subgroup $\Gamma \subset G(\mathbb{Z}[S^{-1}])$, which is torsion free. The number of  $\Gamma$- orbits of proper $\mathbb{Q}$-parabolic subgroups is finite. Let us denote their representative as $\{P_{1},P_{2},...,P_{r}\}$.
We conjugate them by appropriate elements of $G(\mathbb{Q})$ so that the $P_{i}(\mathbb{Q}_{S})$ contain the minimal parabolic subgroup $M_{0,S}A_{0,S}N_{0,S}$. We denote them as $Q_{i,S}= M_{i,S}A_{i,S}N_{i,S}$, and their corresponding conjugating elements as $\delta_{i}\in G(\mathbb{Q})$ (i.e.  $\delta_{i}P_{i}\delta_{i}^{-1}=Q_{i}$). Let $M_{i,S}=M_{Q_{i},S}$, $N_{i,S}=N_{Q_{i},S}$ and $A_{i,S}=A_{Q_{i},S}$. Moreover we put $\Gamma_{i}=\delta_{i}\Gamma \delta_{i}^{-1}$, $\Gamma_{N_{i,S}}=\Gamma_{i}\cap N_{i,S}$ and $\Gamma_{A_{i,S}}=\Gamma_{i}\cap A_{i,S}$.\\


Let $X^*(M(\mathbb{Q}_{S}))_{\mathbb{Q}_{S}}$ be the set of $\mathbb{Q}_{S}$ characters of $M(\mathbb{Q}_{S})$, the Levi subgroup of $P_{S}$. The dual of this space, which can be identified with the Lie algebra of the maximal split part of the center of $M(\mathbb{Q}_{S})$ is   $$\mathfrak{a}_{M(\mathbb{Q}_{S})}=\text{Hom}(X^*(M(\mathbb{Q}_{S}))_{\mathbb{Q}_{S}},\mathbb{R})$$

For $\nu_{S} \in X^*(M(\mathbb{Q}_{S}))_{\mathbb{Q}_{S}}$, we have the following Harish-Chandra homomorphism $H_{M(\mathbb{Q}_S)}$:
$$e^{\langle H_{M(\mathbb{Q}_S)}(m),\nu_{S}\rangle}= \prod_{p \in S}\lvert \nu_{p}(m_{p}) \rvert_{p}.$$
Let $\omega_{S}$ be the irreducible unitary square integrable admissible representation of $M(\mathbb{Q}_{S})$ which is trivial on $A(\mathbb{Q}_{S})$. We define the set of equivalence irreducible classes of $\omega_{S}$ as $\mathcal{E}_{2}(M(\mathbb{Q}_{S}))$. For $\nu_{S} \in X^*(M(\mathbb{Q}_{S}))_{\mathbb{Q}_{S}}\otimes \mathbb{C}=\mathfrak{a}_{M(\mathbb{Q}_{S}),\mathbb{C}}^*$, we can define the following induced representation on $G_{S}$ with parameters $(\omega_{S},\nu_{S})$:\\
\begin{eqnarray*}
\text{Ind}(\omega_{S},\nu_{S})=\prod_{p \in S}\text{Ind}(\omega_{p},\nu_{p}).
\end{eqnarray*}
\subsection{Test Functions} Let $(\tau,V_{\tau})$ be an irreducible $K_{\infty}$-type, i.e. an irreducible finite dimensional representation of $K_{\infty}$. Suppose $d_{\tau}$ and $\chi_{\tau}$ denote the dimension and the character of the above representation respectively.
Let $C_{c}^{\infty}(G(\mathbb{R}),\tau,\tau)$ be the following space of functions
$$\bigg\{\phi_{\infty} : G(\mathbb{R}) \rightarrow \text{End}(V_{\tau}), \phi_{\infty}(k_{1}gk_{2})=\tau(k_{2}^{-1})\phi_{\infty}(g)\tau(k_{1}^{-1})\bigg\}.$$ 

A function $\Phi_{\infty} \in C_{c}^{\infty}(G(\mathbb{R}))$ is called bi-$K_{\infty}$-finite, if the following condition is satisfied:\\
$$\Phi_{\infty}(x)=\int\limits_{K_{\infty}}\int\limits_{K_{\infty}}d_{\tau}\chi_{\tau}(k)\Phi_{\infty}(k^{-1}xk')d_{\tau}\chi_{\tau}({k'}^{-1})dk' dk.$$
A function $\Phi_{\infty}$ is called $K_{\infty}$-central if $\Phi_{\infty}(kxk^{-1})=\Phi_{\infty}(x)$ for all $k \in K_{\infty}$ and for all $x\in G(\mathbb{R})$. We denote the convolution algebra of bi-$K_{\infty}$-finite and $K_{\infty}$-central functions as  $C_{c}^{\infty}(G(\mathbb{R}))_{K_{\infty}}^{K_{\infty}}$. This convolution algebra is isomorphic to the $\text{End}(V_{\tau})$-valued algebra defined above via the following isomorphism : 
\begin{eqnarray*}
     C_{c}^{\infty}(G(\mathbb{R}),\tau,\tau) \cong C_{c}^{\infty}(G(\mathbb{R}))_{K_{\infty}}^{K_{\infty}}\\
     \phi_{\infty} \mapsto \Phi_{\infty}=d_{\tau}\text{Tr}\phi_{\infty}\\
     \int\limits_{K_{\infty}}\Phi_{\infty}(gk)\tau(k)dk=\phi_{\infty}(g) \mapsfrom \Phi_{\infty}(g)
\end{eqnarray*}
At the non-Archimedean places of $S'= S\backslash \infty$ we define the Hecke algebra as the space of compactly supported, locally constant functions. We denote this space as $C_{c}^{\infty}(G(\mathbb{Q}_{S'}))$. We define the co-center of this Hecke algebra as the following quotient:
$$\bar{\mathcal{H}}(G(\mathbb{Q}_{S'})):= \frac{C_{c}^{\infty}(G(\mathbb{Q}_{S'}))}{[C_{c}^{\infty}(G(\mathbb{Q}_{S'})),C_{c}^{\infty}(G(\mathbb{Q}_{S'}))]}.$$
Moreover we choose the functions from this space which are bi-${K'}_{S'}$ invariant, where ${K'}_{S'}$ is an arbitrary compact subgroup of the maximal compact subgroup $K_{S'}$. We denote this subspace as $\bar{\mathcal{H}}({K'}_{S'}\backslash G_{S'}/{K'}_{S'})$. Let $\Phi_{S'} \in \bar{\mathcal{H}}({K'}_{S'}\backslash G_{S'}/{K'}_{S'})$. Hence we can combine the End$(V_{\tau})$-valued function $\phi_{\infty}$ at the Archimedean place with $\Phi_{S'}$ to obtain an endomorphism valued test function on $G_{S}$ and denote the set containing these functions as:
$$C_{c}^{\infty}( G(\mathbb{R}),\tau,\tau)\otimes \bar{\mathcal{H}}({K'}_{S'}\backslash G_{S'}/{K'}_{S'}).$$ 
We would denote the scalar valued counterpart of the above space as  $$C_{c}^{\infty}(G(\mathbb{R}))_{K_{\infty}}^{K_{\infty}}\otimes \bar{\mathcal{H}}({K'}_{S'}\backslash G_{S'}/{K'}_{S'}).$$

Let $L^2(\Gamma \backslash G_{S},V_{\tau})$ be the following set:
$$ \bigg\{f:\Gamma \backslash G_{S} \mapsto V_{\tau}: f(gk_{\infty})=\tau(k_{\infty})^{-1}f(g), (f_1,f_2)=\int_{\Gamma \backslash G_{S}}\langle f_{1}(x),f_{2}(x)\rangle_{V_{\tau}} dx\bigg\}.$$
Here the inner product makes sense as $\text{Vol}(\Gamma \backslash G_{S}) < \infty$. Also as $\Gamma$ is chosen to be a torsion free congruence subgroup, $\Gamma \backslash G(\mathbb{R})$ is a manifold. Elements of  $C_{c}^{\infty}( G(\mathbb{R}),\tau,\tau)\otimes \mathcal{C}({K'}_{S'}\backslash G_{S'}/{K'}_{S'})$ acts on this space via convolution.\\
Let $R$ be the right regular representation of $G(\mathbb{R})$ on $L^2(\Gamma \backslash G(\mathbb{R}))$. Let $\Omega_{G(\mathbb{R})}$ be the Casimir operator in $\mathcal{Z}(\mathfrak{g}_{\mathbb{C}})$, the center of the universal enveloping algebra. Then $-\Omega_{G(\mathbb{R})} \otimes Id$ induces a self adjoint operator $\Delta_{\tau}$ whose restriction to \\
$$L_{\text{cusp}}^2(\Gamma \backslash G(\mathbb{R}),\tau):= (L_{\text{cusp}}^2(\Gamma \backslash G(\mathbb{R}))\otimes V_{\tau})^{K_{\infty}}$$
has pure point spectrum with finite multiplicities. Let us denote them as 
$$0< \lambda_{1}(\tau) < \lambda_{2}(\tau)...$$ with finite multiplicities. Suppose $\mathcal{E}(\lambda_{i}(\tau))$ denotes the eigenspace corresponding to the eigenvalue $\lambda_{i}(\tau)$. We have the counting function as\\
$$N_{\text{cusp}}^{\Gamma}(T,\tau)= \sum_{\lambda_{i}(\tau)\leq \sqrt{T}}\text{dim}(\mathcal{E}(\lambda_{i}(\tau))).$$\\
We can redefine the above counting function by representation theoretic means in the following way: Let $\Pi_{\text{cusp}}(G(\mathbb{Q}_{S}))$ be the set of unitary irreducible cuspidal subrepresentations of the regular representation of $G(\mathbb{Q}_{S})$ on $L_{\text{cusp}}^2(\Gamma \backslash G(\mathbb{Q}_{S})/K_{S'}^')$. Let $\Pi_{\text{cusp}}(G(\mathbb{R}))$ be the set of subrepresentations of the regular representations of $G(\mathbb{R})$ acting on $L_{\text{cusp}}^2(\Gamma \backslash G(\mathbb{Q}_{S})/K_{S'}^')$. Any element $\pi \in \Pi_{\text{cusp}}(G(\mathbb{Q}_{S}))$ can be written as $\pi=\pi_{\infty}\otimes \pi_{S\backslash \infty}$, where $\pi_{\infty} \in \Pi_{\text{cusp}}(G(\mathbb{R}))$. Let $H_{\pi_{\infty}}(\tau)$ be the $\tau$-isotypical subspace of $(\pi_{\infty},H_{\pi_{\infty}})$. Let $H_{\pi_{S\backslash \infty}}^{K'}$ be the subspace of $K_{S'}^'$-fixed vectors in $(\pi_{S\backslash \infty},H_{\pi_{S\backslash \infty}})$. Let $m(\pi_{\infty})$, resp $m(\pi)$, be the multiplicity with which $\pi_{\infty}$, resp $\pi$, occurs as a subrepresentation of $G(\mathbb{R})$, resp $G(\mathbb{Q}_{S})$, on the cuspidal subspace $L_{\text{cusp}}^2(\Gamma \backslash G(\mathbb{Q}_{S})/K_{S'}^')$. Then we have
$$m(\pi_{\infty})=\sum_{\pi^{'}\in \Pi_{\text{cusp}}(G(\mathbb{Q}_{S}))}m(\pi^{'})\text{dim}H_{\pi_{S\backslash \infty}}^{K'}$$
for all $\pi^{'}$ such that $\pi_{\infty}^{'}=\pi_{\infty}$. Suppose $\nu_{\pi}$ denotes the Casimir eigenvalue of $\pi_{\infty}$. Then we take the subcollection $\Pi_{\text{cusp}}(G(\mathbb{Q}_{S}))_{T}$ whose elements satisfies $|\nu_{\pi}|^2\leq T$. Similarly we define $\Pi_{\text{cusp}}(G(\mathbb{R}))_{T}$. Then we have
$$\sum_{\pi_{\infty}\in \Pi_{\text{cusp}}(G(\mathbb{R}))_{T}}m(\pi_{\infty})\text{dim}\text{Hom}_{K_{\infty}}(H_{\pi_{\infty}}(\tau),V_{\tau})= N_{\text{cusp}}^{\Gamma}(T,\tau).$$
\subsection{Fourier Transform} We now define the scalar valued Fourier transform of functions on $C_{c}^{\infty}(G(\mathbb{R}))_{K_{\infty}}^{K_{\infty}}$. Let $P_{\infty}=M_{\infty}^1A_{\infty}N_{\infty}$ be the Langlands decomposition of a standard cuspidal parabolic subgroup of $G(\mathbb{R})$. Choose $\omega_{\infty} \in \mathcal{E}_{2}(M_{\infty}^1)$. Suppose $\theta_{\omega_{\infty}}$ denotes its character. Let $d_{\omega_{\infty}}$ be the formal degree of $\omega_{\infty}$. Let $\tau$ be the double representation of $K_{\infty}$ on $L^2(K_{\infty} \times K_{\infty})$ obtained from $(\tau, V_{\tau})$. Let $\tau_{M_{\infty}}$ be the restriction of $\tau$ to $K_{\infty} \bigcap M_{\infty}$. We let $L_{\omega}^2(M_{\infty},\tau_{M_{\infty}})$ be the set of $\tau_{M_{\infty}}$-spherical functions on $L^2(M_{\infty}) \otimes L^2(K_{\infty}\times K_{\infty})$. The norm in this space is defined as:
$$\lVert \psi \rVert^2 = \int_{M_{\infty}} \int_{K_{\infty}\times K_{\infty}} \|\psi(k_{1}:m:k_{2})\|^2 dk_{1}dk_{2}dm.$$
It can be made into a Hilbert algebra with the multiplication via
$$(\psi_{1}\psi_{2})(k_{1}:m:k_{2})=\int_{M_{\infty}}\int_{K_{\infty}}\psi_{1}(k_{1}:\Tilde{m}:k^{-1})\psi_{2}(k:\Tilde{m}^{-1}m:k_{2}) dk d\Tilde{m}.$$
The Fourier transform of functions $\Phi_{\infty} \in C_{c}^{\infty}(G(\mathbb{R}))$ is defined as in \cite{Ar2}: $$\Phi_{\infty} \mapsto \widehat{\Phi_{\infty}}(\omega_{\infty},\nu_{\infty}) \in L_{\omega}^2(M_{\infty},\tau_{M_{\infty}}),$$ 
where the formula of $\widehat{\Phi_{\infty}}(\omega_{\infty},\nu_{\infty})$ is\\
$$\widehat{\Phi_{\infty}}(\omega_{\infty},\nu_{\infty})(k_{1}:m:k_{2})=d_{\omega_{\infty}} \int\limits_{M_{\infty}}\int\limits_{A_{\infty}}\int\limits_{N_{\infty}}\Phi_{\infty}(k_{1}nam\widetilde{m}k_{2}) \theta_{\omega_{\infty}}(\widetilde{m}^{-1}) e^{(-\nu_{\infty}+\rho_{\infty})\ln(a)}dn da d\widetilde{m}.$$
Next We define the operator valued Fourier transform of $\Phi_{\infty}\in C_{c}^{\infty}(G(\mathbb{R}))_{K_{\infty}}^{K_{\infty}}$. Let $\pi_{\infty}\in \widehat{G(\mathbb{R})}(\tau)$, the unitary irreducible representation of $G(\mathbb{R})$ that contains $\tau$ upon restriction to $K_{\infty}$. Then $\Phi_{\infty} \mapsto \pi_{\infty}(\Phi_{\infty})$ defines the operator valued Fourier transform on the space of endomorphisms of finite dimensional vector space. From Harish-Chandra's sub-representation theorem we know that $\pi_{\infty}$ is isomorphic to an irreducible subrepresentation of a induced representation from an cuspidal parabolic $P_{\infty}$ with parameters $(\omega_{\infty},\nu_{\infty})$. Let us denote the induced representation as $\text{Ind}(\omega_{\infty},\nu_{\infty})$. Then we have the following relation \cite{Ar2}:
$$d_{\omega_{\infty}}\text{Tr}(\text{Ind}(\omega_{\infty},\nu_{\infty})\Phi_{\infty})=\int\limits_{K_{\infty}}\widehat{\Phi_{\infty}}(\omega_{\infty},\nu_{\infty})(k^{-1}:1:k)dk.$$
Let $m_{\pi_{\infty}}(\tau)$ be the multiplicity with which $\tau$ appears in the decomposition of $\pi_{\infty}$ restricted to $K_{\infty}$. Suppose $\pi_{\infty,\tau}(\Phi_{\infty})$ is the restriction of $\pi_{\infty}$ on $\mathcal{H}_{\pi_{\infty}}(\tau)$. Then we can define the spherical Fourier transform \cite{Camp} $\mathcal{F}(\Phi_{\infty})(\pi_{\infty}) \in \text{End}(\mathbb{C}^{m_{\pi_{\infty}}(\tau)})$ as follows:
$$\pi_{\infty,\tau}(\Phi_{\infty})=\mathbbm{1}_{\tau}\otimes \mathcal{F}(\Phi_{\infty})(\pi_{\infty}).$$
Let $\widetilde{\Phi_{\infty}}(x)=\overline{\Phi_{\infty}(x^{-1})}$. Then $\pi_{\infty}(\widetilde{\Phi_{\infty}})=\pi_{\infty}(\Phi_{\infty})^*$, the conjugate transpose of $\pi_{\infty}(\Phi_{\infty})$. Moreover $\text{Tr}\pi_{\infty}(\Phi_{\infty}\star \widetilde{\Phi_{\infty}})=\left|\left|\pi_{\infty}(\Phi_{\infty})\right|\right|_{\text{HS}}^2$. Let $\mu_{\infty}(\omega_{\infty},\nu_{\infty})$ be the Harish-chandra $\mu$ function corresponding to induced parameters $(\omega_{\infty},\nu_{\infty})$. Let $\mathcal{P}$ be the set of associated classes of parabolic subgroups. The Plancherel inversion of $\Phi_{\infty}$ has the following formula:

$$\Phi_{\infty}\star \widetilde{\Phi_{\infty}}(e) =\sum\limits_{\mathcal{P}}n(\mathcal{P})^{-1}\sum\limits_{P \in \mathcal{P}}\sum\limits_{\mathcal{E}_{2}(M_{\infty})}d_{\omega}(\frac{1}{2\pi i})^{q}\int\limits_{i\mathfrak{a}_{\infty}^*}\left|\left|\text{Ind}(\omega_{\infty},\nu)(\Phi_{\infty})\right|\right|_{\text{HS}}^2\mu_{\infty}(\omega_{\infty},\nu)d\nu.$$
Similarly, at the non-Archimedean place $p$ if we assume the induced parameters are $(\omega_{p}\otimes \nu_{p})$, then the the Plancherel Measure $\mu_{p}(\omega_{p})$ is defined over a connected compact manifold $\mathcal{O}_{2}(M(\mathbb{Q}_{p}))$ for each $\omega_{p} \in \mathcal{E}_{2}(M(\mathbb{Q}_{p}))$. We denote by $d_{\omega_{p}}$ the Euclidean measure of the connected compact manifold. Then we have the following Placnherel inversion formula for $\Phi_{p} \in \mathcal{C}(K'_{p}\backslash G(\mathbb{Q}_{p})/K'_{p})$:
$$\Phi_{p}\star \widetilde{\Phi_{p}}(e) =\sum\limits_{\mathcal{P}}n(\mathcal{P})^{-1}\sum\limits_{P \in \mathcal{P}}\sum\limits_{\mathcal{E}_{2}(M_{p})}d_{\omega_{p}}\int\limits_{\mathcal{O}_{2}(M(\mathbb{Q}_{p}))}\left|\left|\text{Ind}(\omega_{p})(\Phi_{p})\right|\right|_{\text{HS}}^2\mu_{p}(\omega_{p})d\omega_{p}.$$

\section{Plancherel Measure and estimates}
In this section we are going to review the explicit formula of Plancherel measure in the case of Reductive Lie group (Real and $p-$adic) and their various estimates which is one of the essential part to prove the main term of the Weyl's law. The references for  this section are \cite{HC3} and \cite{HC'}.\\
\subsection{Real case} We now review some necessary formulas in the Real case.
\subsubsection{product formula}
For this section let us fix a parabolic subgroup $(P,A)$ of $G(\mathbb{R})$, with the corresponding Langlands Decomposition $P=M^1AN$. Let $\omega_{\infty} \in \mathcal{E}_{2}(M^1)$, be a square integrable unitary irreducible class of representation of $M^1$. Let $\mu(\omega_{\infty},\nu_{\infty})$ for $\omega_{\infty} \in \mathcal{E}_{2}(M^1)$ and $\nu_{\infty} \in \mathfrak{a}_{\mathbb{C}}^*$ be the Harish-Chandra $\mu$ function of the pair $(G(\mathbb{R}),P)$. Denote by $\Sigma$, the set of all roots of $(P,A)$. A root $\alpha \in \Sigma$ is called reduced if $k\alpha \notin \Sigma$ for $0\leq k< 1/2$. Let $\Phi$ be a set of all reduced roots. For any $\alpha \in \Phi$ put \\
\begin{eqnarray}\label{2}
\mathfrak{n}_{\alpha}= \bigoplus_{k\alpha:{k\geq1}}\mathfrak{n}(\alpha),
\end{eqnarray}
where $\mathfrak{n}(\alpha)= \{x \in \mathfrak{g}: [H,X]=\alpha(H)X,\forall H \in \mathfrak{a}\}$. Let $N_{\alpha}$ be the analytic subgroup corresponding to $\mathfrak{n}_{\alpha}$.\\
Let $\sigma_{\alpha}$ be the hyper-plane given by $\alpha=0$ in $\mathfrak{a}$. Let $Z_{\alpha}$ be the centralizer of $\sigma_{\alpha}$. We put $M_{\alpha}=\prescript{0}{}{Z_{\alpha}}$, $A_{\alpha}= M_{\alpha} \cap A$ and $\theta(N_{\alpha})=\bar{N_{\alpha}}$. Then we can define the following parabolic subgroups with their corresponding Langlands decompositions\\
$$\prescript{*}{}{P_{\alpha}}= M^1A_{\alpha}N_{\alpha},\quad \bar{\prescript{*}{}{P_{\alpha}}}= M^1A_{\alpha}\bar{N_{\alpha}},\quad P=M^1AN.$$
We can now define the product formula for the Harish-chandra $\mu$ function.\\
Suppose $\mu(\omega_{\infty},\nu_{\infty}^{\alpha})$ denotes the corresponding $\mu(\omega_{\infty},\nu_{\infty})$ for the group $(M_{\alpha},\prescript{*}{}{P_{\alpha}})$. Here $\nu_{\infty} \in \mathfrak{a}_{\mathbb{C}}^*$ and $\nu_{\infty}^{\alpha}$ denotes the restriction of $\nu_{\infty}$ to $A_{\alpha}$, and $\omega \in \mathcal{E}_{2}(M^1)$. Then for a suitable constant $C_{G}$ depending on $G$ we have the following product formula \cite[Theorem 12, p. 145]{HC'}\\
\begin{eqnarray}\label{3}
\mu(\omega_{\infty},\nu_{\infty})=C_{G}\prod_{\alpha \in \Phi} \mu(\omega_{\infty},\nu_{\infty}^{\alpha}).
\end{eqnarray}

Here we have that prk $M_{\alpha}=0$, prk$\prescript{*}{}{P_{\alpha}}=1$. 
\subsubsection{Explicit formula}
When prk $G=0$, and prk $P=1$ we have the following two possibilities\\
\begin{itemize}
    \item $\mathcal{E}_2(G)\neq \emptyset$
    \item $\mathcal{E}_2(G)= \emptyset$
\end{itemize}
Here the first condition is equivalent to Rank $G=$Rank $K$. We write down the formula for the $\mu$ function in each case.\\
\begin{itemize}
\item \cite[Theorem 1, section 24]{HC3} We consider the second condition first. Let us introduce some notations. Let $Q$ be the set of positive roots of $(\mathfrak{g},\mathfrak{h})$, where $\mathfrak{h}$ is a $\theta$-stable Cartan subalgebra of $\mathfrak{g}$. Now $Q$ is the union of three disjoint parts $Q_{I},Q_{R},Q_{C}$, set of imaginary, real and complex roots respectively. Let $H_{\alpha}$ be the unique element in $\mathfrak{h}$ such that $(H_{\alpha},H)=\alpha(H)$, for all $H \in \mathfrak{h}$, where $(,)$ denotes the Killing form. With respect to Cartan involution we have the decomposition $\mathfrak{g}=\mathfrak{k}\bigoplus \mathfrak{p}$. Moreover we have $\mathfrak{h}=\mathfrak{h}_{I}\oplus \mathfrak{h}_{R}$, and a fixed parabolic subgroup with Langlands decomposition $P=M^1e^{\mathfrak{h}_{R}}N.$ We put\\
$$\widetilde{w_{I}}=\prod_{\alpha \in Q_{I}} H_{\alpha},\quad \widetilde{w_{R}}=\prod_{\alpha \in Q_{R}} H_{\alpha},\quad \widetilde{w_{+}}=\prod_{\alpha \in Q_{C}} H_{\alpha}.$$
Suppose $\mathfrak{h}_I$ is a cartan subalgebra of $\mathfrak{k}$. Then the second condition is satisfied.\\
From the theorem of Harish-Chandra \cite[section 23,Theorem 1]{HC3} we know that $\omega_{\infty} \in \mathcal{E}_2(M^1)$ corresponds to an element in orbit of  $H^{*'}_{I}$ under the action of $W(M^1/H_{I})$, where $H^{*'}_{I}$ is a subspace of $H^{*}_{I}$, which is the Cartan subgroup of $M^1$ in $K \cap M^1$(we call this element $\lambda$ which lies in the lie algebra of $H_{I}^{*'}$). Let $a^{*}$ be the element in $H^{*'}_{I}$ that corresponds to $\omega_{\infty}$. Put $\lambda=\lambda(a^{*}) \in i\mathfrak{h}_{I}^{*'}$. Then we have the following formula of the $\mu$ function:
\begin{eqnarray}\label{4}
\mu(\omega_{\infty},\nu_{\infty})=C \widetilde{w_{+}}(\lambda+\nu_{\infty})=C \prod_{\alpha \in Q_{C}}\left|(\lambda+\nu_{\infty},\alpha)\right|.
\end{eqnarray}
Here the constant $C$ depends on $G(\mathbb{R})$ and $M^1$.
In this case one can see that if $Q_{R}$ is empty \cite[2.3.5, p. 58]{Wallach}, then dim $N= | Q_{C}|$. So $\mu$ is a polynomial in $\nu_{\infty}$ of degree dim $N$. 
Hence as $t \to \infty$ we have 
$$\mu(\omega_{\infty},t\nu_{\infty}) \sim C_{\nu_{\infty}}t^{\text{dim}(N)},$$ for a non zero positive constant $C_{\nu_{\infty}}$. 

Therefore when $G=M_{\alpha}$, $N=N_{\alpha}$, $e^{\mathfrak{h}_{\mathbf{R}}}=A_{\alpha}$ we have, $$\mu(\omega_{\infty},t\nu_{\infty}^{\alpha}) \sim C_{\nu^{\alpha}_{\infty}}t^{\text{dim}(N_{\alpha})} \quad \text{as}\quad t \to \infty.$$
\item  \cite[Sec. 36, p. 190]{HC3} Next we consider the case of rank $G$= rank $K$.  Let $\mathfrak{h}$ be a $\theta$-stable Cartan subalgebra of $\mathfrak{g}$ with $\mathfrak{h}=\mathfrak{h}_{I}\oplus \mathfrak{h}_{R}$, dim $(e^{\mathfrak{h}_{\mathbf{R}}})=1$, and a fixed parabolic subgroup with Langlands decomposition $P=M^1e^{\mathfrak{h}_{R}}N$. So dim $N= 1+|Q_{C}|.$
\par
Let $H_{I}$ be the analytic subgroup corresponding to $\mathfrak{h}_{I}$. Let $Q$ be the set of positive roots of $(\mathfrak{g},\mathfrak{h})$ which is the disjoint union of imaginary$(Q_{I})$, real$(Q_{R})$ and complex roots$(Q_{C})$. Let us denote by $\alpha$ the unique root in $Q_{R}$. For $a^* \in H^*_{I}$, let us define:\\
$$\mu_{0}(a^*,\nu_{\infty}):= d(a^*)^{-1} \text{Tr}\left(\frac{\pi i\nu_{\infty}^{\alpha}\sinh{\pi i\nu_{\infty}^{\alpha}}}{\cosh{\pi i\nu_{\infty}^{\alpha}}-\frac{(-1)^{\rho_{\alpha}}}{2}(\sigma_{a^*}(\gamma)+\sigma_{a^*}(\gamma^{-1}))}\right).$$
Here $d(a^*)$ is the degree, $\nu_{\infty}^{\alpha}=2\frac{(\nu_{\infty},\alpha)}{(\alpha,\alpha)}$, $\rho_{\alpha}= 2\frac{(\rho,\alpha)}{(\alpha,\alpha)} \in \mathbb{Z}$, where $\rho$ is the half of the sum of the positive roots of $(\mathfrak{g},\mathfrak{h})$), $\sigma_{\alpha}$ is the irreducible representation of $H_{I}$ whose character is $a^*$ and $\gamma$ is a fixed element in $H_{I}$. So the above expression has the form:\\
$$u(z)=\frac{z\sinh{\pi z}}{\cosh{\pi z}+k},$$
for a fixed real number $k \geq -1$.
\par
As $\omega_{\infty} \in \mathcal{E}_{2}(M^1)$ corresponds to an element $a^* \in H^{*'}_{I}$, which we denote by $\lambda=\lambda(a^{*}) \in \mathfrak{h}_{I}^{*'}$ as the corresponding element in the Lie algebra of $H_{I}$.\\
As before we define:
$$\widetilde{{w}_{+}}(\omega_{\infty}:\nu_{\infty})=\prod_{\alpha \in Q_{C}}\left|(\lambda+\nu_{\infty},\alpha)\right|.$$
Moreover, we define
$$\mu_{0}(\omega_{\infty},\nu_{\infty})= \frac{1}{|W(M^1/H_{I})|}\sum_{W(M^1/H_{I})}\mu_{0}(sa^*,\nu_{\infty}).$$
With these notations in mind we can write the Harish-Chandra $\mu$-function as follows
\begin{eqnarray}\label{5}
\mu(\omega_{\infty},\nu_{\infty})=C \lvert \alpha \rvert \mu_{0}(\omega_{\infty},\nu_{\infty})\widetilde{{w}_{+}}(\omega_{\infty}:\nu_{\infty}).
\end{eqnarray}
We can show that $u(tz) \sim C_{z}t$, as $t \to \infty$, when $z \in \mathbb{R}$ and $C_{z}>0$. On the other hand $\widetilde{{w}_{+}}(\omega_{\infty}:\nu_{\infty})$ is a polynomial in $\nu_{\infty}$ with degree $\text{dim}(N)-1$. Therefore we obtain the same asymptotic expression
$$\mu(\omega_{\infty},t\nu_{\infty}) \sim C_{\nu} t^{\text{dim}N} \quad \text{as} \quad t \to \infty.$$
With the previous notation(i.e. when $G=M_{\alpha}$, $N=N_{\alpha}$, $e^{\mathfrak{h}_{\mathbf{R}}}=A_{\alpha}$) we have, $\mu(\omega_{\infty},t\nu_{\infty}^{\alpha}) \sim C_{\nu^{\alpha}} t^{\text{dim}(N_{\alpha})}$ as $t \to \infty$.\\
\end{itemize}

\subsubsection{Asymptotic estimate}
Because we have the polar decomposition $G=KP$, where $P=MAN$ is the Langlands decomposition, we have\\
$$\text{dim}(G/K)= \text{dim}(\frac{M^1}{K\cap M^1})+\text{dim}(A)+\text{dim}(N).$$

Therefore in the case where $P=P_{0}$ is the minimal parabolic we have  $$\text{dim}(G/K)=\text{dim}(A_{0})+\text{dim}(N_{0})=\text{dim}(A_{0})+\sum_{\alpha \in \Phi}\text{dim}(N_{0,\alpha}),$$ and we have the following estimate of the density of the Plancherel measure:
\begin{eqnarray}\label{6}
\int_{\nu_{\infty} \in i\mathfrak{a}_{0,\mathbb{R}}^*: (\nu_{\infty},\nu_{\infty})\leq t^2}\mu(\omega_{\infty},\nu_{\infty})d\nu \sim C_{G} C_{\nu_{\infty}}t^{\text{dim}(G/K)}, \quad \text{as} \quad  t \to \infty. 
\end{eqnarray}
Now for the pair $(M_{\alpha},A_{\alpha})$ we have polynomial bound. Here we invoke \cite[Thm 1, Sec. 25]{HC3}. The theorem states that $\mu$ can be extended to the whole complex plane meromorphically. Moreover there exists $C,r \geq 0$ such that:
$$\lvert\mu(\omega_{\infty},\nu_{\infty})\rvert \leq C(1+\lVert \text{Im}(\nu_{\infty})\rVert)^r.$$
Our job is to find an explicit value of $r$ in the inequality mentioned in Harish-Chandra's paper. 
\par
The case when $\mathcal{E}_{2}(M_{\alpha}) = \emptyset$, we have that $\mu(\omega_{\infty},\nu_{\infty}^{\alpha})$ is a polynomial in $\nu_{\infty}$ of degree  $\text{dim}(N_{\alpha})$. So in this scenario we can take $r= \text{dim}(N_{\alpha})$. And similarly for the other case we can arrive at the same estimate, as $\sup_{z \in \mathbb{C}} \lvert u(z)\rvert \leq (1+\lvert z\rvert)$, whenever $z$ is real. And the other part, namely $\widetilde{{w}_{+}}(\omega_{\infty} : \nu_{\infty})$, is a polynomial in $\nu_{\infty}$ of degree $\text{dim}(N_{\alpha})-1$.\\
Hence combining with the product formula mentioned above we conclude that for some $C'>0$
\begin{eqnarray}\label{7}
\mu(\omega_{\infty},\nu_{\infty}) \leq C'(1+|\lvert \nu_{\infty} \rvert|)^{\text{dim}(N)}.
\end{eqnarray}
\textbf{Remark}: Important point to note here the constant $C'$ does not really matter in terms of finding out the main term of the Weyl's law. Only constant that could matter is $C_{G}$.  \cite[Sec. 6.3]{LV}). 
\subsection{The $p-$adic case}
The comment that we are going to make is due to \cite[P. 355]{HC-p}. The Plancherel measure in this case is defined on $\mathcal{E}_{2}(M_{p})$, as evident from the formulas written above. For this case this set is compact,(can be denoted as $\sqcup_{\omega_{p} \in \mathcal{E}_{2}(M_{p})} \mathcal{O}_{\omega_{p}}$) hence the asymptotic estimate will not change if we are to consider the plancherel inversion formula for the group $G_{S}$.\\
Hence combining the above two subsection we arrive at the estimate that \\
\begin{eqnarray}\label{8}
\int_{\nu \in i\mathfrak{a}_{0,\infty}^*\times \mathcal{E}_{2}(M_{p}): (\nu_{\infty},\nu_{\infty})\leq t^2}\mu(\omega,\nu)d\nu d\omega_{p} \sim \alpha(G) t^{\text{dim}(G_{\infty}/K_{\infty})}.
\end{eqnarray}

\section{Condition for purely cuspidal image}
In this section we provide the necessary condition on the space of scalar valued test functions so that the image of convolution operator on scalar valued $K_{\infty}$-finite automorphic forms only consists of cuspidal $K_{\infty}$-finite automorphic forms . We closely follow  \cite[Prop. 3, second proof]{LV}.\\
First we need some preparation.
We recall a couple of lemmas due to Harish-Chandra regarding vanishing condition of Schwartz functions.\\
Let us recall some of the notations mentioned already in the preliminaries.
Let $Q=M_{Q}N_{Q}$ be a standard parabolic subgroup of $G(\mathbb{R})$ and $G(\mathbb{Q}_{p})$. Let  $\mathcal{C}(G(\mathbb{R}),\tau)$ be the Harish-Chandra Schwartz space of vector valued function which are $\tau$-spherical. These are functions from $G(\mathbb{R})$ to $V_{\tau}\subset L^2(K_{\infty}\times K_{\infty})$, where $V_{\tau}$ is viewed as a double representation of $\tau$. The action of $\tau$ could be defined as
$$\tau(k)\phi(k_{1}:g:k_{2})\tau(k')=\phi(k_{1}k:g:k'k_{2}).$$ 
Let $\mathcal{C}_{\text{cusp}}(M_{Q},\tau_{M_{Q}})$ be the space of functions which are cuspidal, $\tau_{M_{Q}\cap K_{\infty}}$-spherical, $\mathfrak{Z}(M_{Q})$-finite and $A_{Q}$-invariant. The $L^2$-completion of this space is generated by the square integrable matrix coefficients of finitely many classes of isomorphic unitary irreducible discrete series representations of $M_{Q}$ which are $A_{Q}$-invariant. For more information about this space see \cite[Chp. I,Sec. 2]{Ar}.
Define:
$$\phi_{\infty}^{Q}(la)=\int_{N_{Q}}\phi_{\infty}(nla) dn, \quad \phi_{\infty} \in \mathcal{C}(G(\mathbb{R}),\tau),$$
for all $la \in M_{Q}^1 A_{Q}$.
Moreover we write $\phi_{\infty}^{Q} \sim 0$ if the following holds:
$$\int_{M_{Q}/A_{Q}}\Big(f(l),\phi_{\infty}^{Q}(la)\Big)dl=0, \quad \forall f \in \mathcal{C}_{\text{cusp}}(M_{Q},\tau_{M_{Q}^1})$$
for all $a \in A_{Q}$, where $( , )$ denotes the inner product in the vector space $V_{\tau}$ (which is viewed as a double representation space of the action of $\tau$ on $L^2(K_{\infty} \times K_{\infty})$).
Then by \cite[vol IV, p. 149]{HC-p} we have the following.

\begin{lemma}[Archimedean case]
Let $\phi_{\infty}$ be an element in $\mathcal{C}(G(\mathbb{R}),\tau)$ such that $\phi_{\infty}^{Q} \sim 0$ for all parabolic subgroups $Q(\mathbb{R}) \subset G(\mathbb{R})$. Then $\phi_{\infty}=0$.
\end{lemma}
We now modify the above conditions for scalar valued bi-$K_{\infty}$-finite functions, by using the ideas mentioned in \cite[vol IV, p. 175]{HC-p}. We denote the corresponding scalar valued function as $\Phi_{\infty}$. Hence again we write:\\
we will write $\Phi_{\infty}^{Q} \sim 0$ if \\
$$\int_{M_{Q}^1}f(l)\Phi_{\infty}^{Q}(la) dl=0, \quad \forall a \in A_{Q}.$$
The above integral is a function defined on $A_{Q}$, the split part of the center of $M_{Q}$. Hence if $\Phi_{\infty}$ is a compactly supported smooth function on $G(\mathbb{R})$, then the integral above is also a compactly supported function defined on $A_{Q}$. 
We recall some characterization of parabolic subgroups of standard Levi subgroups due to Harish-chandra\\
\begin{lemma}
Let $p \in S$. There is an one to one correspondence between parabolic subgroup $P(\mathbb{Q}_{p})$ of $G(\mathbb{Q}_{p})$ which are contained in $Q(\mathbb{Q}_{p})$, and parabolic subgroups $^{*}P(\mathbb{Q}_{p})$ of $M_{Q}(\mathbb{Q}_{p})$. The correspondence are as follows:
If $Q(\mathbb{Q}_{p})=M_{Q}(\mathbb{Q}_{p})N_{Q}(\mathbb{Q}_{p})$ and  $P(\mathbb{Q}_{p})=M_{P}(\mathbb{Q}_{p})N_{P}(\mathbb{Q}_{p})$ are the corresponding Levi decompositions, then $^{*}P_{Q}=P(\mathbb{Q}_{p})\cap M_{Q}(\mathbb{Q}_{p})=M_{P}(\mathbb{Q}_{p})N_{Q}^{P}(\mathbb{Q}_{p})$ is the corresponding Levi decomposition, where $A_{P}(\mathbb{Q}_{p})=A_{Q}^{P}(\mathbb{Q}_{p})A_{Q}(\mathbb{Q}_{p})$, $N_{P}(\mathbb{Q}_{p})=N_{Q}^{P}(\mathbb{Q}_{p})N_{Q}(\mathbb{Q}_{p})$.
\end{lemma}
\begin{lemma}
Let $\Phi \in C_{c}^{\infty}(G(\mathbb{R}))_{K_{\infty}}^{K_{\infty}}\otimes \bar{\mathcal{H}}(G(K_{S^{'}}^{'}\backslash \mathbb{Q}_{S'}/K_{S^{'}}^{'}))$. Assume that: 
\begin{eqnarray}\label{9}
\text{Ind}(\omega_{S},\nu_{S})(\Phi)=0,
\end{eqnarray}
for all $i$, for all parabolic $P_{S} \subset Q_{i,S}$ whose Archimedean part is the chosen minimal parabolic subgroup, for all $\omega_{S}$, equivalence classes of unitary irreducible representation of $M(\mathbb{Q}_{S})$ whose Archimedean component is discrete series and non-Archimedean components are cuspidal representations, such that $\omega_{\infty} \subset \tau|_{K_{\infty}\cap M_{\infty}}$, for all $\nu_{S} \in \mathfrak{a}_{P_{0,S},\mathbb{C}}^*$ satisfying $\nu_{S}\lvert_{\Gamma_{A_{i,S}}}=1$,
then $\Phi$ satisfies the following equation for all $i$, for all $k_{1},k_{2} \in K_{S}$ and for all $m \in M_{i,S}$:
\begin{eqnarray}\label{10}
\sum_{\Gamma_{A_{i,S}}}\int_{N_{i,S}}\Phi(k_{1}n\gamma mk_{2}) dn=0.
\end{eqnarray}
\end{lemma}
\begin{proof}
The equation \eqref{9} implies
\begin{equation*}
\int\limits_{M(\mathbb{Q}_{S})}\int\limits_{N(\mathbb{Q}_{S})}\int\limits_{{K_{S}}}\Phi(k_{1}mnk_{2})\omega_{S}(m^{-1})e^{(-\nu_{S}+\rho_{S})H_{P_{S}}(m^{-1})}dk_{2}dndm =0,
\end{equation*}
for all $k_{1} \in K_{S}$, for all $P(\mathbb{Q}_{S})=P_{S}\subset Q_{i,S}=Q_{i}(\mathbb{Q}_{S})$, and for all $\nu_{S} \in \mathfrak{a}_{P_{0,S},\mathbb{C}}^{*}$ such that $\nu_{S}|_{\Gamma_{A_{i,S}}}=1$. (we will suppress the iteration i in our discussion that follows)
Hence we can drop the integration on $K_{S}$ to get:
\begin{equation*}
\int\limits_{M(\mathbb{Q}_{S})}\int\limits_{N(\mathbb{Q}_{S})}\Phi(k_{1}mnk_{2})\omega_{S}(m^{-1})e^{(-\nu_{S}+\rho_{S})H_{P_{S}}(m^{-1})}dndm =0.
\end{equation*}
Breaking the group $N(\mathbb{Q}_{S})$ as product of $N_{Q}^{P}(\mathbb{Q}_{S})$ and $N_{Q}(\mathbb{Q}_{S})$ we have the following:
\begin{equation}\label{11}
\int\limits_{M(\mathbb{Q}_{S})}\int\limits_{N_{Q}^{P}(\mathbb{Q}_{S})}\int\limits_{N_{Q}(\mathbb{Q}_{S})}\Phi(k_{1}mn'n''k_{2})\omega_{S}(m^{-1}) e^{(-\nu_{S}+\rho_{S})H_{P_{S}}(m)}
dn''dn'dm=0.
\end{equation}

As $^{*}P_{Q}(\mathbb{Q}_{S})=M(\mathbb{Q}_{S})N_{Q}^{P}(\mathbb{Q}_{S}) \subset M_{Q,S}$, is a parabolic subgroup of $M_{Q,S}$, 
we have $\Gamma_{A,S} \subset M(\mathbb{Q}_{S})$ and $\Gamma_{A,S}$ centralizes $M(\mathbb{Q}_{S})$. Hence \eqref{11} implies
\begin{equation}\label{12}
\begin{aligned}
\int\limits_{\frac{M(\mathbb{Q}_{S})}{\Gamma_{A,S}}}\int\limits_{N_{Q}^{P}(\mathbb{Q}_{S})}\int\limits_{N_{Q}(\mathbb{Q}_{S})}\sum\limits_{\Gamma_{A,S}}\Phi(k_{1}m_{1}\gamma n'n''k_{2})\omega_{S}(m_{1}^{-1})\\e^{(-\nu_{S}+\rho_{S})H_{P_{S}}(m_{1})}dn''dn'dm_{1}=0.
\end{aligned}
\end{equation}
We will apply Fubini's theorem to change the order of the integral on $N_{Q}(\mathbb{Q}_{S})$ and sum on $\Gamma_{A,S}$ (as $\Phi$ is compactly supported the integral above is convergent). Moreover we can break down the set $\frac{M(\mathbb{Q}_{S})}{\Gamma_{A,S}}$ into product of $\frac{M(\mathbb{Q}_{S})}{A(\mathbb{Q}_S)}$ and $\frac{A(\mathbb{Q}_S)}{\Gamma_{A,S}}$. Hence we can think of the above integral as the Fourier transform of the following integral:
$$\int\limits_{\frac{M(\mathbb{Q}_{S})}{A(\mathbb{Q}_S)}}\int\limits_{N_{Q}^{P}(\mathbb{Q}_{S})}\int\limits_{N_{Q}(\mathbb{Q}_{S})}\sum\limits_{\Gamma_{A,S}}\Phi(k_{1}m_{1}\gamma n'n''k_{2})f_{S}(m_{1})dn''dn'dm_{1},$$
where $f_{S}$ is the product of coefficient of discrete series (at the Archimedean place) and cuspidal representations (at the non-Archimedean places) of $M(\mathbb{Q}_{S})$.
Therefore, using the injectivity of Fourier transform on functions defined on $M_{Q,S}$ which are compactly supported modulo the central direction of $M_{Q,S}$ (which holds for non-Archimedean case by combining \cite[Th. 25]{B}, \cite{BDK} and \cite[Sec. 5.7]{CH}), eq. 12 implies
\begin{eqnarray}\label{13}
\sum_{\Gamma_{A_{i,S}}}\int_{N_{i,S}}\Phi(k_{1}n\gamma mk_{2}) dn=0.
\end{eqnarray}
\end{proof}
\begin{lemma}
If $\Phi \in C_{c}^{\infty}(G(\mathbb{R}))_{K_{\infty}}^{K_{\infty}}\otimes \bar{\mathcal{H}}(G(K_{S^{'}}^{'}\backslash \mathbb{Q}_{S'}/K_{S^{'}}^{'}))$ satisfies the \eqref{10} then $\Phi$ maps the elements in $L_{\text{loc}}^{1}(\Gamma \backslash G_{S})$ to $L_{\text{cusp}}^{2}(\Gamma \backslash G_{S})$ as a convolution operator.
\end{lemma}
\begin{proof}
We fix $1\leq i \leq r$ and let $\Psi \in L^2(\Gamma \backslash G_{S})_{K_{\infty}}$ and  $\Psi_{i}(g)=\Psi(\delta_{i}g)$. Then it is easy to see that $\Psi_{i} \in L^2(\Gamma_{i} \backslash G_{S})_{K_{\infty}}$. To get the purely cuspidal image we need the following:\\
For $\Phi=\Phi_{\infty}\Phi_{S\backslash \infty}$, where $\Phi_{\infty}\in C_{c}^{\infty}(G(\mathbb{R}))_{K_{\infty}}^{K_{\infty}}$, $\Phi_{S\backslash \infty}\in \bar{\mathcal{H}}(G(K_{S^{'}}^{'}\backslash \mathbb{Q}_{S'}/K_{S^{'}}^{'}))$ and $\Psi_{i} \in L^2(\Gamma_{i} \backslash G_{S})_{K_{\infty}}$,
$$\int_{\Gamma_{N_{i,S}}\backslash N_{i,S}} \Psi_{i} \star \Phi(nx) dn =0.$$ This is equivalent to:
$$\int_{\Gamma_{N_{i,S}} \backslash N_{i,S}}\int_{G_S} \Phi(y^{-1}nx) \Psi_{i}(y) dy dn=0.$$
Now we break down the integral on $G_S$ as an integral on $\Gamma_{i} \backslash G_S$ and a discrete sum on $\Gamma_{i}$ to get the equivalent relation:
$$\int\limits_{\Gamma_{N_{i,S}} \backslash N_{i,S}}\sum_{\Gamma_{i}}\int_{\Gamma_{i}\backslash G_S} \Phi(y^{-1}\gamma^{-1} nx)\Psi(y) dy dn=0,$$ for all $x,y$. This is equivalent to:
$$\int\limits_{\Gamma_{N_{i,S}} \backslash N_{i,S}}\sum_{\Gamma_{i}}\Phi(y^{-1}\gamma^{-1}nx) dn=0.$$ 
After swapping the integral and the sum (as the sum over $\Gamma_{i}$ is locally finite, and $\Phi$ is compactly supported, we can use the Monotone Convergence Theorem or Dominated Convergence Theorem) and rearranging the domains, we get a sum over $(\Gamma_{i} \bigcap N_{i,S}) \backslash \Gamma_{i}$ and an integral on $N_{i,S}$ 
Therefore, we arrive at the equivalent condition: 
\begin{eqnarray}\label{14}
\sum_{\Gamma_{N_{i,S}} \backslash \Gamma_{i}}\int_{N_{i,S}} \Phi(y^{-1}\gamma nx) dn=0.
\end{eqnarray}
We replace $x$ and $y$ with their respective Iwasawa decompositions, i.e. $x=m_{1}n_{1}k_{1}$ and $y=m_{2}n_{2}k_{2}$. Moreover, we can write the sum over $\frac{\Gamma}{\Gamma \bigcap N_{i,S}}$ as sum over ${\Gamma \bigcap A_{i,S}}=\Gamma_{A_{i,S}}$ cosets. Therefore as $M_{i,S}$ centralizes $A_{i,S}$, we can write $k_{2}^{-1}n_{2}^{-1}m_{2}^{-1}\gamma n m_{1}n_{1}k_{1}=k_{2}^{-1}n_{2}^{-1}\gamma m_{2}^{-1} n m_{1}n_{1}k_{1}$. As $M_{i,S}$ normalizes $N_{i,S}$, with a modular factor we have $$k_{2}^{-1}n_{2}^{-1}\gamma m_{2}^{-1}n m_{1}n_{1}k_{1}=k_{2}^{-1}n_{2}^{-1}n'\gamma m_{2}^{-1}  m_{1}n_{1}k_{1}.$$ But as the modular factor, a scalar, only depends on $m_{2}$, and $\gamma$ belongs to a discrete subgroup, we can ignore that above. Therefore we can finally write the argument of $\Phi$ as $k_{2}^{-1}n_{2}^{-1}n'n''\gamma m_{2}^{-1} m_{1}k_{1}=k_{2}n\gamma m k_{2}$. Hence, we can rewrite the condition in \eqref{14} as follows: For all $k_{1},k_{2} \in K_{S}$ and for all $m\in M_{i,S}$ \cite[(4.6)]{LV}\\
\begin{eqnarray}\label{15}
\sum_{\Gamma_{A_{i,S}}}\int_{N_{i,S}}\Phi(k_{1}n\gamma mk_{2}) dn=0.
\end{eqnarray}
Consequently \eqref{10} is the sufficient condition.
\end{proof}
In the next step we need to find a non-zero combined test functions on $G_{S}$, which is bi-$K_{\infty}$-finite,$K_{\infty}$-central and compactly supported at the Archimedean place and a function from the Hecke algebra at the non-Archimedean places that satisfies the above conditions. For parabolic subgroups $P_{S}$, whose Archimedean component is the chosen standard minimal parabolic subgroup, \eqref{9} would become as follows: For all $\nu_{S} \in \mathfrak{a}_{P_{S}}^* \supset \mathfrak{a}_{Q_{i,S}}^*$ such that whenever for all i, $\nu_{S}|_{\Gamma_{A_{i,S}}}=1$, we have
\begin{eqnarray}\label{16}
\text{Ind}_{P_{S}}^{G_{S}}(\omega_{S},\nu_{S})(\Phi)=0,
\end{eqnarray}
for all $\omega_{S}$ discrete series representation of $M_{P,S}$ such that $\tau|_{M_{0,\infty}} \supset \omega_{\infty}$. By the description of arithmetic tori \cite[Thm. 5.12]{PR}, $\nu_{S}$ should have the property that $\nu_{p}=\nu_{q}$, for all $p,q \in S$. Let $\mathfrak{Z}(G(\mathbb{Q}_{p}))$ be the ring of regular functions defined on union of Benrstein components $\Omega(G(\mathbb{Q}_{p}))$ (\cite[2.3.1]{MT}). Combining for all $p \in S\backslash \infty$ we define $\mathfrak{Z}(\mathbb{Q}_{S\backslash \infty})$ to be the set of Fourier transform of Bernstein center for $G(\mathbb{Q}_{S\backslash \infty})$. Let $z_{i}$ be the elements in the Bernstein center for each $Q_{i}(\mathbb{Q}_{S\backslash \infty})$ \cite[2.2.1]{MT}. Let $\hat{z_{i}} \in \mathfrak{Z}(G(\mathbb{Q}_{p}))$. We can form the test function
$$\Phi=\Phi_{\infty}\prod_{i}(z_{i}\star \mathbbm{1}_{K'}).$$
We have to find a regular function $R$ defined on $\Omega(G(\mathbb{Q}_{S\backslash \infty}))$ such that $R(\nu_{S\backslash \infty})=0$, whenever $\nu_{p}=\nu_{q}$, for all $p,q \in S$ and $\nu_{p} \in \mathfrak{a}_{P_{S}}^*$. We can find a polynomial that satisfies this property for every $P_{S} \subset Q_{i,S}$. 
Hence, we could apply the Arthur's Paley-Wiener Theorem at the Archimedean place and matrix Paley-Wiener Theorem for Hecke alegebra by Bernstein \cite[Thm. 25]{B} at the non-archimedean place to construct a non-zero test function $\Phi$.

\section{Partial trace formula}
In this section we write the partial trace formula. 
We choose a test function whose Archimedean component is a $\tau-$spherical function belonging to the convolution algebra $C_{c}^{\infty}(G(\mathbb{R}),\tau,\tau)$, satisfying the identity  $$\phi(k_{1}gk_{2})=\tau(k_2)^{-1}\phi(g)\tau(k_1)^{-1}$$ and a scalar valued function at the non-archimedean places from $\bar{\mathcal{H}}(G(K_{S^{'}}^{'}\backslash \mathbb{Q}_{S'}/K_{S^{'}}^{'}))$, the Hecke algebra. Suppose it also satisfies the condition of cuspidality described in the previous section.  It acts on $\Gamma-$invariant $L^2$ eigensection ($K_{\infty}$-finite,$K_{S'}^{'}$-fixed) $e_{\lambda}(x)$ (orthonormal with respect to the inner product mentioned in the introduction) of the Casimir operator(defined for sections of vector bundles), with the eigenvalue parameter defined as $\lambda$. We define the convolution action as follows:  
\begin{equation*}
\begin{aligned}
e_{\lambda} \star \phi(x)     & =\int_{G_{S}}\phi(y^{-1}x)e_{\lambda}(y) dy\\
                       & = \int_{\Gamma \backslash G_{S}}\sum_{\gamma^{-1}\in \Gamma} \phi(y^{-1}\gamma^{-1}x)e_{\lambda}(y) dy\\
                       & = \int_{\Gamma \backslash G_{S}} K(x,y)e_{\lambda}(y) dy\\
                       & = \int_{\Gamma \backslash G_{S} /K_{\infty}} K(x,y)e_{\lambda}(y) dy
\end{aligned}    
\end{equation*}
In the last equation we have used the fact that $K(x,y)e_{\lambda}(y)$ is $K_{\infty}-$invariant on $y$.
Then the spectral expansion of $K(x,y)$ can be written as follows:
$$K(x,y)=\sum_{\lambda,\mu}(e_{\lambda} \star \phi, e_{\mu})e_{\mu}(x)\otimes e_{\lambda}(y)^*,$$
where $e_{\mu}(y)^*$ denotes the dual vector which acts on $f(y)$ through the pairing $\langle f(y),e_{\mu}(y)\rangle_{V_{\tau}}$ on the fiber $(E_{\tau})_{y}$ \cite[ (7.3)]{D}.
If we let $x=y$, then the spectral side will have the following form:\\
$$K(x,x)=\sum_{\lambda,\mu}(e_{\lambda} \star \phi, e_{\mu})e_{\mu}(x)\otimes e_{\lambda}(x)^*.$$
Therefore, taking the trace on both sides we get
$$\text{Tr}K(x,x)=\sum_{\lambda}(e_{\lambda} \star \phi, e_{\lambda})e_{\lambda}(x)\otimes e_{\lambda}(x)^*.$$
Consider a compact subset $\Omega \subset \Gamma \backslash G_{S}$, whose measure is arbitrarily close to Vol$(\Gamma \backslash G_{S})$. We take the pre-image of $\Omega$ in $G_{S}$ and call it $\Tilde{\Omega}$.
Unwinding the sum on the left hand side we get\\
$$TrK(x,x)=Tr\phi(e)+\sum_{\gamma \in Z}\text{Tr}\phi(x^{-1}\gamma x),$$ 
The set $Z$ will have the following form :$$Z= \Big( \Gamma \backslash \{e\}\Big)\bigcup \Big(xgx^{-1}: x \in \Tilde{\Omega}, \text {x lies in support of Tr}(\phi)\Big).$$ 
The cardinality of $Z$ would be finite, and would depend only on $\Tilde{\Omega}$ and the support of \text{Tr}$(\phi)$.\\
Integrating both sides over $\Tilde{\Omega}$, we obtain the following:
$$\int_{\Tilde{\Omega}} TrK(x,x) dx  \leq \int_{\Gamma \backslash G_{S}/K_{\infty}} TrK(x,x) dx=\sum_{\lambda}(e_{\lambda} \star \phi, e_{\lambda})$$
To make sure we have a self adjoint convolution operator we need $\phi(x)=\overline{\phi(x^{-1})^T}$. To achieve the self-adjointness we replace $\phi$ with $\phi \star \widetilde{\phi}$, where $\widetilde{\phi}(x)=\overline{\phi(x^{-1})^T}$. Hence, the right hand side of the above inequality becomes $\sum_{\lambda}(e_{\lambda} \star \phi, e_{\lambda} \star \phi)$.\\
Now by a theorem of Gelfand, Graev and Piatetski-Shapiro \cite[Prop. 3.2.3]{bump} which states that the convolution operator on  the scalar-valued automorphic forms is a compact operator. Therefore we obtain:
$\sum_{\lambda}(e_{\lambda} \star \phi, e_{\lambda} \star \phi) < \infty$.
Hence, we have
\begin{eqnarray}\label{17}
\text{Tr}(\phi\star \widetilde{\phi}(e))\text{Vol}(\Omega)+\sum_{\gamma \in Z}\int_{\Omega}\text{Tr}(\phi\star \Tilde{\phi})(x^{-1}\gamma x) \leq \sum\limits_{\lambda}(e_{\lambda} \star \phi, e_{\lambda} \star \phi).
\end{eqnarray}
We now give a representation theoretic interpretation of $\sum_{\lambda}(e_{\lambda} \star \phi, e_{\lambda} \star \phi)$. Let $\pi_{\infty} \in \Pi_{\text{cusp}}(G(\mathbb{R}),\tau)$ be the Archimedean part of irreducible unitary representation $\pi_{S}$ which appears as a subrepresentation of right regular representation of $G(\mathbb{Q}_{S})$ on $L^2_{\text{cusp}}(\Gamma \backslash G(\mathbb{Q}_{S}),\tau)$ with multiplicities $m(\pi_{\infty})$, and let $H_{\pi_{\infty}}$ be the corresponding Hilbert space. Let $H_{\pi_{\infty}}(\tau)$ be the $\tau-$isotypic subspace. Then using \cite[Thm. 3.3]{BM} we have: 
\begin{eqnarray}\label{18}
\sum_{\lambda}(e_{\lambda} \star \phi, e_{\lambda} \star \phi)=\sum\limits_{\Pi_{\text{cusp}}(G(\mathbb{R}),\tau)}m(\pi_{\infty})\bigg(\sum\limits_{i=1}^{m}(e_{i}\star \phi,e_{i}\star \phi)\bigg),
\end{eqnarray}
where $m=\text{dim}(\text{Hom}_{K_{\infty}}(\mathcal{H}_{\pi_{\infty}}(\tau),V_{\tau}))$.

\section{An approximation lemma}
In this section we find a family of test functions $$H_{S,t}=H_{\infty,t}\cdot\mathbbm{1}_{K'},$$ for $0\leq t<1$ that satisfy certain approximations. For the rest of the section and beyond we will write $S\backslash \infty=S'$. Here $K'=K_{S\backslash \infty}^{'}$. We prove a slight generalization of \cite[Lemmma 2]{LV} below.\\
Let $\mathfrak{h}=i\mathfrak{h}_{K_{\infty}}+\mathfrak{a}_{0,\infty}$ be the Cartan subalgebra of $\mathfrak{U(\mathfrak{g_{\infty}})}$. Let $\mathfrak{h}_{\mathbb{C}}=\mathfrak{h}\otimes \mathbb{C}$ be the complexification of the Cartan subalgebra. Let $\mathfrak{h}_{\mathbb{C}}^*$ be the dual of the Cartan subalgebra. We fix $0<\epsilon <1$. By \cite[(5.12)-(5.15)]{LV} we know there exist a non-empty open set of Schwarz functions $\psi$ defined on cylinders
$$\{\lambda_{\infty} \in \mathfrak{h}_{\mathbb{C}}^*: \lvert \text{Re}(\lambda_{\infty})\rvert \leq a\}$$ 
that satisfy the following conditions:
\begin{itemize}
    \item $0\leq \psi(\lambda_{\infty}) < 1, \quad \text{when} \quad \lVert \lambda_{\infty} \rVert \leq 1, \lambda_{\infty} \in \mathfrak{h}_{K_{\infty}}^*+i\mathfrak{a}_{0,\infty}^*$.
    \item $ \int\limits_{i\mathfrak{a}_{0,\infty}^*} \left|\psi(\nu_{K_{\infty}}+\nu_{\infty})-\chi(\nu_{K_{\infty}}+\nu_{\infty})\right|(1+\lVert \nu_{\infty} \rVert)^{\text{dim}(N_{0})}d\nu_{\infty} \leq \epsilon$, for fixed $\nu_{K_{\infty}}\in \mathfrak{h}_{K_{\infty},\mathbb{C}}^*$ such that $\text{Re}(\nu_{K_{\infty}})$ is bounded.
    \item  $\sup\limits_{\lVert \lambda_{\infty} \rVert >1}(1+\lVert \lambda_{\infty} \rVert)^{d+1}\left|\psi(\lambda_{\infty})\right| \leq \epsilon$.
\end{itemize}
Here $\chi(\nu_{K_{\infty}}+\nu_{\infty})$ denotes the characteristic function of the sphere $\lVert \nu_{K_{\infty}}+\nu_{\infty} \rVert \leq 1$. Without loss of generality we may assume that $\psi$ can be extended to a holomorphic function on $\mathfrak{h}_{\mathbb{C}}^{*}$, as the Fourier transform of $\psi$ has compact support. Let $\mathbb{O}_{\mathbb{C}}$ be the orthogonal group of $\mathfrak{h}_{\mathbb{C}}^*$ with respect to the inner product $\langle,\rangle$. This inner product is induced from the Killing form on $\mathfrak{h}_{\mathbb{C}}$. By averaging we can make $\psi$ as $\mathbb{O}_{\mathbb{C}}-$ invariant function. Therefore $\psi(\lambda_{\infty})$ depends only on $\langle \lambda_{\infty},\lambda_{\infty}\rangle$. Let $d_{\omega_{\infty}}$ denote the degree of equivalent classes of square integrable irreducible representations $\omega_{\infty}$ of $M_{0,\infty}$. Using Frobenius Reciprocity we see that for the case of minimal parabolic $M_{0,\infty}$, we have  $$[\text{Ind}(\omega_{\infty},\nu_{\infty})|_{K_{\infty}}:\tau]=[\tau|_{M_{0,\infty}}:\omega_{\infty}].$$ Therefore,  $$\sum\limits_{\omega_{\infty}\in \mathcal{E}_{2}(M_{0,\infty})}d_{\omega_{\infty}}[\tau|_{M_{0,\infty}}:\omega_{\infty}]=\sum\limits_{\omega_{\infty}\in \mathcal{E}_{2}(M_{0,\infty})}d_{\omega_{\infty}}[\text{Ind}(\omega_{\infty},\nu_{\infty})|_{K_{\infty}}:\tau]=d_{\tau}.$$ 
Put $m_{\omega_{\infty}}=[\tau|_{M_{0,\infty}}:\omega_{\infty}]$. Let $C_{c}^{\infty}(\mathfrak{h})$ be the set of compactly supported smooth functions (i.e. set of multipliers for the convolution algebra $C_{c}^{\infty}(G(\mathbb{R}))_{K_{\infty}}^{K_{\infty}}$). Then by the Euclidean Paley-Wiener Theorem there exists $\zeta \in C_{c}^{\infty}(\mathfrak{h})$ such that its Laplace-Fourier transform $\hat{\zeta}(\lambda_{\infty})$ satisfies:
$$\hat{\zeta}(\lambda_{\infty})=\psi(\lambda_{\infty}), \quad \forall \lambda_{\infty} \in \mathfrak{h}_{\mathbb{C}}^*.$$
There exists a function $H_{\infty}^{\sharp} \in C_{c}^{\infty}(G(\mathbb{R}))_{K_{\infty}}^{K_{\infty}}$ such that $\text{Ind}(\omega_{\infty},\nu_{\infty})(H_{\infty}^{\sharp})= \text{Ind}(\omega_{\infty},\nu_{\infty})(d_{\tau}\chi_{\tau})$ \cite[Lemma 1.3.2]{GV}. Now using the Arthur's theorem on multiplier we can choose a family of functions  $H_{\infty,t,\zeta}^{\sharp} \in C_{c}^{\infty}(G(\mathbb{R}))_{K_{\infty}}^{K_{\infty}}$, such that their operator valued Fourier transforms are $$\text{Ind}(\omega_{\infty},\nu_{\infty})(H_{\infty,t,\zeta})=\hat{\zeta}(\nu_{\omega_{\infty}}+t\nu_{\infty})\text{Ind}(\omega_{\infty},\nu_{\infty})(d_{\tau}\chi_{\tau}),$$ for $0< t\leq1$. Let $H_{S,t,\zeta}^{\sharp}=H_{\infty,t,\zeta}^{\sharp} \cdot\mathbbm{1}_{K'}$, for $K'$ an arbitrarily chosen open compact subgroup of $G_{S'}$. Then 
\begin{eqnarray*}
\left|\left|\widehat{H_{\infty,t,\zeta}^{\sharp}}(\omega_{\infty},\nu_{\infty})\right|\right|^2
=\int\limits_{K_{\infty}}\left|\widehat{H_{\infty,t,\zeta}^{\sharp}}(\omega,\nu)(1:1:k)\right|^2dk=d_{\omega}\left|\left|\text{Ind}(\omega_{\infty},\nu_{\infty})(H_{\infty,t,\zeta}^{\sharp})\right|\right|_{\text{HS}}^{2}.
\end{eqnarray*}
Therefore, from the above choice of Schwartz function we have $$\left|\left|\widehat{H_{\infty,t,\zeta}^{\sharp}}(\omega_{\infty},\nu_{\infty})\right|\right|^{2}=d_{\tau}d_{\omega}m_{\omega_{\infty}}\lVert\psi(\nu_{\omega_{\infty}}+t\nu_{\infty})\rVert^2.$$ 
The following estimate will be instrumental in proving the main estimate in Weyl's law. Let $0<\epsilon <1$ and choose $H_{\infty,t,\zeta}^{\sharp}$ that depending on $\epsilon$.
\begin{lemma}
There exists $C_{1}>0$ such that for sufficiently small $0< t \leq1$ and for the minimal Parabolic $P_{0,\infty}=M_{0,\infty}A_{0,\infty}N_{0,\infty}$ we have
\begin{eqnarray*}
\left|t^{d}\sum_{\omega \in \mathcal{E}_{2}(M_{0,\infty})} \int_{i\mathfrak{a}_{0,\infty}^*}\left|\left|\widehat{H_{\infty,t,\zeta}^{\sharp}}(\omega_{\infty},\nu_{\infty})\right|\right|^2\mu(\omega_{\infty},\nu_{\infty})d\nu-d_{\tau}^2\alpha(G_{\infty})\right| \leq C_{1}\epsilon.
\end{eqnarray*}
\end{lemma}
\begin{proof}
Recall the Plancherel inversion formula at the real place\\
\begin{eqnarray}\label{19}
f\star \Tilde{f}(1)= \sum_{\mathcal{P}}n(\mathcal{P})^{-1}\sum_{P\in \mathcal{P}}\sum_{\omega \in \mathcal{E}_{2}(M_{\infty})}(\frac{1}{2\pi i})^{q}\int_{i\mathfrak{a}_{M_{\infty}}^*}\left|\left| \widehat{f}(\omega,\nu)\right|\right|^2 \mu(\omega,\nu)d\nu.
\end{eqnarray}
Here, $\mathcal{P}$ denotes the associated classes pf parabolic subgroups. The integer $q$ denotes the dimension of respective $i\mathfrak{a}_{M_{\infty}}^{*}$. We are interested on the summand that corresponds to the minimal parabolic. For the sum and integral involving the minimal parabolic subgroup $P_{0,\infty} =M_{0,\infty}A_{0,\infty}N_{0,\infty}$, and $\text{dim}(i\mathfrak{a}_{0,\infty}^{*})=r$ we have
$$\limsup_{t \to 0}\left| t^{d}\sum_{\omega_{\infty} \in \mathcal{E}_{2}(M_{0,\infty})}(\frac{1}{2\pi i})^{r}\int_{i\mathfrak{a}_{0,\infty}^*}\lVert\widehat{{H_{\infty,t,\zeta}^{\sharp}}}(\omega_{\infty},\nu_{\infty})\rVert^{2} \mu(\omega_{\infty},\nu_{\infty}) d\nu_{\infty} - d_{\tau}^2\alpha(G_{\infty})\right|$$
$$=\limsup_{t \to 0}\left| t^{d}\sum_{\omega_{\infty} \in \mathcal{E}_{2}(M_{0,\infty})}(\frac{1}{2\pi i})^{r}\int_{i\mathfrak{a}_{0,\infty}^*}\lVert\widehat{{H_{\infty,t,\zeta}^{\sharp}}}(\omega_{\infty},\nu_{\infty})\rVert^{2} \mu(\omega_{\infty},\nu_{\infty}) d\nu_{\infty} - d_{\tau}^2\alpha(G_{\infty})\right|$$
$$\leq \left|\limsup_{t \to 0} t^{d}\sum_{\omega_{\infty} \in \mathcal{E}_{2}(M_{0,\infty})}(\frac{1}{2\pi i})^{r}\int_{i\mathfrak{a}_{0,\infty}^*}d_{\tau}d_{\omega}m_{\omega_{\infty}}\left|\lVert\psi(\nu_{\omega_{\infty}}+t\nu_{\infty})\rVert^{2} -\chi(\nu_{\omega_{\infty}}+t\nu_{\infty})\right|\mu(\omega_{\infty},\nu_{\infty}) d\nu_{\infty}\right|$$
$$\leq \abs*{\limsup_{t \to 0} t^{d}\sum_{\omega_{\infty} \in \mathcal{E}_{2}(M_{0,\infty})}(\frac{1}{2\pi i})^{r}\int_{i\mathfrak{a}_{0,\infty}^*}d_{\tau}d_{\omega}m_{\omega_{\infty}}\left|\lVert\psi(\nu_{\omega_{\infty}}+\nu_{\infty})\rVert^{2} -\chi(\nu_{\omega_{\infty}}+\nu_{\infty})\right|
\mu(\omega_{\infty},t^{-1}\nu_{\infty})d(t^{-1}\nu_{\infty})}$$
$$\leq\left|\sum_{\omega_{\infty} \in \mathcal{E}_{2}(M_{0,\infty})}(\frac{1}{2\pi i})^{r}\int_{i\mathfrak{a}_{0,\infty}^*}d_{\tau}d_{\omega}m_{\omega_{\infty}}\left|\lVert\psi(\nu_{\omega_{\infty}}+\nu_{\infty})\rVert^{2} -\chi(\nu_{\omega_{\infty}}+\nu_{\infty})\right|\Big(1+\lVert\nu_{\infty} \rVert\Big)^{\text{dim}(N_{0})}  d(\nu_{\infty})\right|$$
$$\leq C_{1} \epsilon.$$

In the last step we use the second condition of $\psi$ defined in the beginning of this section.

\end{proof}

\section{Plancherel Inversion and test functions}
In this section we describe the choice of test functions. We start by recalling a result of Camporesi, which identifies the endomorphism valued convolution algebra with scaler valued functions.
\begin{prop*}\cite[Prop 2.1]{Camp}
Let $\tau$ be the irreducible $K_{\infty}$-type as before of dimension $d_{\tau}$. then the endomorphism valued convolution algebra isomorphic to scalar valued bi-$K_{\infty}$-finite, $K_{\infty}$-central function space. The anti-isomorphism is given by the following map
\begin{eqnarray}\label{20}
f \mapsto F=d_{\tau}\text{Tr}(f).
\end{eqnarray}
Moreover it satisfies the following relations:
$$d_{\tau}\text{Tr}(f_{1}\star f_{2})= d_{\tau}\text{Tr}(f_{2})\star d_{\tau}\text{Tr}(f_{1}),$$
$$d_{\tau}\chi_{\tau}\star F= F = F\star d_{\tau}\chi_{\tau}.$$

\end{prop*}

\subsection{Test Functions}
The following steps will describe our test functions. We closely follow \cite[Lemma 2,3]{LV}.\\
\ding{108} We choose a function $\Phi_{S}^{\sharp} = \Phi_{\infty}^{\sharp}\Phi_{S'}$, where $\Phi_{\infty}^{\sharp} \in C_{c}^{\infty}(G(\mathbb{R}))_{K_{\infty}}^{K_{\infty}}$ and $\Phi_{S'} \in \bar{\mathcal{H}}(G(K_{S^{'}}^{'}\backslash \mathbb{Q}_{S'}/K_{S^{'}}^{'}))$, 
so that the cuspidality condition holds true, i.e. $\Phi_{S}^{\sharp}$ satisfy $(8)$. By the isomorphism in $(20)$ we have a function $\phi_{\infty} \in C_{c}^{\infty}(G(\mathbb{R}),\tau,
\tau)$ so that $d_{\tau}\text{Tr}\phi_{\infty}=\Phi_{\infty}^{\sharp}$. Let $\phi_{S}$ be the product of $\phi_{\infty}$ by $\Phi_{S'} \in \bar{\mathcal{H}}(G(K_{S^{'}}^{'}\backslash \mathbb{Q}_{S'}/K_{S^{'}}^{'}))$.\\
\ding{108} Next we choose a family of functions $$h_{\infty,t,\zeta} \in C_{c}^{\infty}(G(\mathbb{R}),\tau,\tau) \quad \text{for} \quad 0< t \leq 1.$$ From the properties mentioned in the previous section, we can choose an entire Schwartz function $\widehat{H_{\infty,\zeta}}(\omega,\nu)$ that satisfies the Lemma 6. We form a family $h_{\infty,t,\zeta}$ for $0< t \leq1$, so that $d_{\tau}\text{Tr}h_{\infty,t,\zeta}=H_{\infty,t,\zeta}^{\sharp}$. We multiply $h_{\infty,t,\zeta}$ with $\mathbbm{1}_{K'}$, and call this function $h_{S,t,\zeta}$.\\
\ding{108} Finally, choose a sequence $\Phi_{n,S}^{\sharp}$ that satisfies the condition of cuspidality. Let $Z_{\infty} \in C_{c}^{\infty}(\mathfrak{h})$ be an element in the set of Archimedean multipliers. Then by Euclidean Paley-Wiener Theorem, $\hat{Z}_{\infty}$ is bounded on the set $\{\lambda_{\infty} \in \mathfrak{h}_{\mathbb{C}}^*: \text{Im}(\lambda_{\infty})=0\}$. If we choose $\Phi_{S}$ so that it satisfies the condition of cuspidality, then the fourier transform of elements of the Bernstein center at the non-Archimedean places is bounded on the set of unitary unramified characters. Suppose the bound is $B>0$. Following \cite[p. 245-246]{LV} we construct such a sequence. Let
$$P_n(Z_{S})=1-\Big(1-\frac{(Z_{S})^2}{B^2}\Big)^{n}.$$
Let $\Phi_{S}^{\sharp}=Z_{S}\star f_{S}$, where $\lVert\pi_{\infty}(f_{\infty})\rVert_{\text{HS}}^2=1$ and $f_{S\backslash \infty}=\mathbbm{1}_{K'}$. As the multipliers on the space of bi-$K_{\infty}-$finite compactly supported smooth functions on $G(\mathbb{R})$ and multipliers on space of locally constant functions on $G(\mathbb{Q}_{S'})$ are equipped with convolution, (thought of as a multiplication) we can define a sequence $\Phi_{n,S}^{\sharp}= P_{n}(Z_{S})\star f_{S}\star \widetilde{f_{S}}$. $P_{n}$ will satisfy the following properties
\begin{itemize}
    \item $P_{n}(0)=0$
    \item $\text{Ind}(\omega_{S},\nu_{S})(\Phi_{n,S}^{\sharp})=P_{n}(\hat{Z}_{S})(\text{Ind}(\omega_{S},\nu_{S})(f_{S})\text{Ind}(\omega_{S},\nu_{S})(f_{S})^*).$\\ 
    Notice that here multiplication of Fourier transforms are defined as \cite[part II,pp. 1.1 ]{Ar1}. Therefore $\Phi_{n,S}^{\sharp}$ will satisfy \eqref{9}.
\end{itemize}    
Let $\phi_{n,S}$ be the endomorphism valued test functions corresponding to $\Phi_{n,S}^{\sharp}$ as $\phi_{n,S}$. We apply the partial trace formula on the family of test functions $\phi_{n,S}\star h_{S,t,\zeta}$. Write $\phi_{n,S,t,\zeta}=\phi_{n,S}\star h_{S,t,\zeta}$. We obtain
\begin{eqnarray}\nonumber
&d_{\tau}\text{Tr}((\phi_{n,S,t,\zeta}\star \widetilde{\phi_{n,S,t,\zeta}})(e))\text{Vol}(\Omega)+d_{\tau}\sum_{\gamma \in Z}\int_{\Omega}\text{Tr}(\phi_{n,S,t,\zeta}\star\widetilde {\phi_{n,t,S}})(x^{-1}\gamma x)\\\nonumber                   &\leq d_{\tau}\sum\limits_{\lambda}(e_{\lambda} \star \phi_{n,S,t,\zeta}\star \widetilde{\phi_{n,S,t,\zeta}}, e_{\nu})\\\label{21}
&=d_{\tau}\sum\limits_{\lambda}\sum\limits_{-\nu_{\pi_{\infty}}=\lambda}m(\pi)\text{Tr}(\pi(\text{Tr}(\phi_{n,,S,t,\zeta}\star \widetilde{\phi_{n,S,t,\zeta}}))).
\end{eqnarray}

\subsection{Plancherel Inversion}
We now recall the Plancherel Theorem at Archimedean place as in \cite[part II, (2.1)]{Ar1} \\
\begin{eqnarray*}
&\widetilde{\Phi_{n,\infty,t,\zeta}^{\sharp}}\star \Phi_{n,\infty,t,\zeta}^{\sharp} (e)\widetilde{\Phi_{n,S'}}\star \Phi_{n,S'}(e)=\sum_{\mathcal{P} \in \text{Cl}(G_{\infty})}n(\mathcal{P})^{-1}\sum_{P \in \mathcal{P}}\sum_{\omega_{\infty} \in \mathcal{E}_2(M_{\infty})}(\frac{1}{2\pi i})^{q}\\
&\int_{i\mathfrak{a_{\infty}^*}}\left|\left|\text{Ind}(\omega_{\infty},\nu_{\infty})(\Phi_{n,\infty,t,\zeta}^{\sharp})\right|\right|_{\text{HS}}^2\mu(\omega_{\infty},\nu_{\infty}) d\nu_{\infty}\widetilde{\Phi_{n,S'}}\star \Phi_{n,S'}(e).
\end{eqnarray*}
We have the following convergences as $n\to \infty$
$$\text{Ind}(\omega_{S'})(\Phi_{n,S'}) \to \text{Ind}(\omega_{S'})(\mathbbm{1}_{K'}), \quad \text{and} \quad \lVert\text{Ind}(\omega_{\infty},\nu_{\infty})(\Phi_{n,\infty}^{\sharp})\rVert \to 1.$$
Now if we take the limit inside the norm (due to continuity) and inside the Fourier transformation (due to isometry)\cite[p. 4719]{Ar2} the above integrand converges to
$$\left|\left|\widehat{H_{\infty,t,\zeta}^{\sharp}}(\omega_{\infty},\nu_{\infty})\right|\right|^2\mu(\omega_{\infty},\nu_{\infty}).$$
Therefore, we see that integrand corresponding to the minimal parabolic summand in the Plancherel Formula can be divided into two sets $X=\{\nu \in i\mathfrak{a}_{0,\infty}^*:P_{n}(\hat{Z}_{S})\leq \epsilon \}$ and its complement $X^c$. As we take $\lim{\epsilon \to 0}$, the set $X$ will have measure 0, and on $X^c$ the integrand will become $\lVert \widehat{H_{\infty,t,\zeta}^{\sharp}}(\omega_{\infty},\nu_{\infty})\rVert^2\mu(\omega_{\infty},\nu_{\infty})$.
From the discussion above, it is clear that the following estimate will be enough for us to arrive at the main term as $\limsup{t \to 0}$.
\begin{lemma}
For all $n$, there exists $C_{1} > 0$ such that
\begin{align*}
\abs*{t^{d}\sum_{\mathcal{P} \in \text{Cl}(G_{\infty})}n(\mathcal{P})^{-1}\sum_{P \in \mathcal{P}}\sum_{\omega_{\infty} \in \mathcal{E}_2(M_{\infty})}(\frac{1}{2\pi i})^{q}
\int\limits_{i\mathfrak{a}_{\infty}^*}\left|\left|\widehat{{H_{\infty,t,\zeta}^{\sharp}}}(\omega_{\infty},\nu_{\infty})\right|\right|^{2} \mu(\omega_{\infty},\nu_{\infty})d\nu_{\infty}\\
\widetilde{\Phi_{n,S'}}\star \Phi_{n,S'}(e)- d_{\tau}^2\alpha(G)} \leq C_{1}\epsilon.
\end{align*}
\end{lemma}
\begin{proof}
We can ignore the terms related to the $p-$adic Plancherel formula as the tempered parameters in this case are finite disjoint unions of compact orbifolds, hence those terms will be automatically bounded. Therefore, we only concentrate on the Archimedean part. We have 
$$\left|t^{d}\sum_{\mathcal{P} \in \text{Cl}(G_{\infty})}n(\mathcal{P})^{-1}\sum_{P \in \mathcal{P}}\sum_{\omega_{\infty} \in \mathcal{E}_2(M_{\infty})}(\frac{1}{2\pi i})^{q}\int_{i\mathfrak{a}_{\infty}^*}\left|\left|\widehat{{H_{\infty,t,\zeta}^{\sharp}}}(\omega_{\infty},\nu_{\infty})\right|\right|^{2} \mu(\omega_{\infty},\nu_{\infty}) d\nu_{\infty} - d_{\tau}^2\alpha(G)\right|$$
$$=\left| t^{d} \times (\text{non-minimal terms})+t^{d}\sum_{\omega_{\infty} \in \mathcal{E}_2(M_{0,\infty})}(\frac{1}{2\pi i})^{r}\int_{i\mathfrak{a}_{0,\infty}^*}\left|\left|\widehat{{H_{\infty,t,\zeta}^{\sharp}}}(\omega_{\infty},\nu_{\infty})\right|\right|^{2} \mu(\omega_{\infty},\nu_{\infty}) d\nu_{\infty} - d_{\tau}^2\alpha(G)\right|$$

The Plancherel density corresponding to the non-minimal parabolic subgroups will have the following asymptotic estimate. For some integer $l<d$, we have
$$\int_{i\mathfrak{a}_{P,\infty}^*} \mu(\omega_{\infty},t^{-1}\nu_{\infty})d(t^{-1}\nu_{\infty}) \sim t^{-l} \quad \text{as} \quad t\to 0.$$
Therefore, the non-minimal terms will tend to 0 as $t\to 0$.
And the approximation for the other term was dealt with in Lemma 6.
\end{proof}
Therefore, We see from \eqref{21} the main term corresponding to the trivial conjugacy class is asymptotic to 
$$d_{\tau}^2\alpha(G)\text{Vol}(\Omega)t^{-d} \quad \text{as} \quad t \to 0.$$

\section{Bounds for the non-trivial classes}
To get the estimates for non-trivial conjugacy classes on the geometric side we write the Fourier inversion formula of Harish-chandra with respect to Eisenstein integrals. To this end we use the formula $(1.1)$ in \cite[Chap. III Sec. 1]{Ar}. It gives
\begin{eqnarray*}
&H_{\infty,t}(x)|_{(1,1)}\\         
&= \sum\limits_{\mathcal{P}}\lvert {\mathcal{P}\rvert}^{-1} \sum_{P \in \mathcal{P}}{\lvert W(\mathfrak{a}_{P})\rvert}^{-1}\int\limits_{i\mathfrak{a}_{\infty}^{*}}E_{P}(x_{\infty},\mu_{P}(\nu_{\infty})\widehat{H_{\infty,P}}(t\nu_{\infty}),\nu_{\infty})_{(1:1)} d\nu_{\infty},
\end{eqnarray*}
where $\mathcal{P}$ denotes an associated class of parabolic subgroups, the function $H_{\infty,t}$ lies in $C_{c}^{\infty}(G(\mathbb{R}),\tau)$. Note that $C_{c}^{\infty}(G(\mathbb{R}),\tau)$ is isomorphic to $C_{c}^{\infty}(G(\mathbb{R}))_{K_{\infty}}$ via the relation:
$$H_{\infty,t}(x)|_{k_{1},k_{2}}=H_{\infty,t}(k_{1}xk_{2}).$$
We concentrate on the part of the above series and integral corresponding to the minimal parabolic. We obtain the following inequality for the summand corresponding to minimal parabolic $P_{0,\infty}$:
\begin{eqnarray}\label{22}\nonumber
&\int\limits_{i\mathfrak{a}_{0,\infty}^{*}}\lvert E_{P_{0}}(x_{\infty},\mu_{P_{0}}(\nu_{\infty})\widehat{H_{\infty,P_{0}}}(t\nu_{\infty}),\nu_{\infty})|_{(1:1)}\rvert d\nu_{\infty}   &  \\   
&\leq \int\limits_{i\mathfrak{a}_{0,\infty}^{*}}\int\limits_{K_{\infty}}\lvert \widehat{H_{\infty,P_{0}}}(t\nu)(1:m(kx):1) \rvert \mu_{P_{0}}(\nu) \lvert e^{(\nu+\rho)H(kx_{\infty})}\rvert dk d\nu   &  
\end{eqnarray}
The right hand side of the above inequality is bounded by
$$\int\limits_{i\mathfrak{a}_{0,\infty}^{*}}\Bigg(\lVert \widehat{H_{\infty,P_{0}}}(t\nu) \rVert \mu_{P_{0}}(\nu) \int\limits_{K_{\infty}}\lvert e^{(\nu+\rho)H(kx)}\rvert dk\Bigg) d\nu.$$
Moreover, we have $\int_{K_{\infty}}\lvert e^{(\nu+\rho)H(kx)}\rvert dk d\nu \leq C_{1}(1+\lVert \nu\rVert)^{-1/2}$ when $x \notin K_{\infty}$ and lies in a fixed compact set\cite[Lemma 3]{LV}.
Now making a change of variable $\nu_{\infty} \to \frac{\nu_{\infty}}{t}$, using the Paley-Wiener bound of $\lVert \widehat{H_{\infty,P}}(\nu_{\infty}) \rVert$, and using the bound of the Plancherel Measure from  \eqref{7} from section 3, we obtain the following inequality:
\begin{eqnarray}\label{23}
\int\limits_{i\mathfrak{a}_{0,\infty}^*}\lvert E_{P}(x_{\infty},\mu_{P}(\nu)\widehat{H_{\infty,P}}(t\nu)(1:1),\nu)\rvert d\nu \leq C_{2}t^{-d+1/2}.
\end{eqnarray}\label{nontrivial}
Hence, we have the bound
\begin{eqnarray}
\left|H_{\infty,t}(x)|_{(1,1)}\right| \leq C_{2}t^{-d+1/2}. 
\end{eqnarray}
Now we apply the above bound for the Archimedean part of the integrand corresponding to the non-trivial conjugacy class of $\gamma \in \Gamma$ in \eqref{21}. Notice that similar to \cite[(6.3)]{LV} we can assume that the support of $d_{\tau}\text{Tr}(\phi_{n,S,t,\zeta}\star \widetilde{\phi_{n,S,t,\zeta}})$ that lies inside $K_{\infty}$ will have measure zero when projected onto $G(\mathbb{R})$. Moreover, 
$$d_{\tau}\text{Tr}(\phi_{n,\infty,t,\zeta}\star \widetilde{\phi_{n,\infty,t,\zeta}})=\Phi_{n,\infty,t,\zeta}^{\sharp}\star \widetilde{\Phi_{n,\infty,t,\zeta}^{\sharp}}.$$ 
This is a bi-$K_{\infty}$-finite function. Hence we can apply the above discussion. Arguing as in \cite[pg. 243]{LV} for the lower bound of Weyl's law, we can see that the terms corresponding to a non-trivial conjugacy class in \eqref{21} is bounded by $c|Z|t^{-d+1/2}$. Notice that as $\Gamma$ injects into $G_{S}$ diagonally, the cardinality of $Z$ is finite. Moreover the $L^1$ norm of  $\Phi_{n,S'}\mathbbm{1}_{K'}$ are bounded by a constant for all $n$. Therefore it follows from \eqref{21} that
\begin{eqnarray}\nonumber
\left|d_{\tau}\text{Tr}((\widetilde{\phi_{n,S,t,\zeta}}\star \phi_{n,S,t,\zeta})(e))\text{Vol}(\Omega)+d_{\tau}\sum_{\gamma \in Z}\int_{\Omega}\text{Tr}(\widetilde{\phi_{n,S,t,\zeta}}\star \phi_{n,S,t,\zeta})(x^{-1}\gamma x)\right|<\sum_{\lambda}C_{\phi_{n,S,t,\zeta}}\nonumber
\end{eqnarray}
or,
\begin{eqnarray*}
\left|d_{\tau}\text{Tr}((\widetilde{\phi_{n,S,t,\zeta}}\star \phi_{n,S,t,\zeta})(e))\text{Vol}(\Omega)\right|-c|Z|t^{-d+1/2}<\sum_{\lambda}C_{\phi_{n,S,t,\zeta}}\nonumber
\end{eqnarray*}
or,
\begin{eqnarray*}
\left|t^{d}\widetilde{\Phi_{n,S,t,\zeta}^{\sharp}}\star \Phi_{n,S,t,\zeta}^{\sharp}(e)\text{Vol}(\Omega)\right|-c|Z|t^{1/2}<t^{d}\sum_{\lambda}C_{\phi_{n,S,t,\zeta}}.
\end{eqnarray*}
Here, $C_{\phi_{n,S,t,\zeta}}=(e_{\lambda}\star \phi_{n,S,t,\zeta},e_{\lambda}\star \phi_{n,S,t,\zeta})$.
From \eqref{21} and \cite[Lemma 3.3]{BM} we have that
\begin{eqnarray}\label{25}
\sum\limits_{\lambda}(e_{\lambda}\star \phi_{n,S,t,\zeta},e_{\lambda}\star \phi_{n,S,t,\zeta})= \sum\limits_{\Pi_{\text{cusp}}(G(\mathbb{Q}_{S}),\tau)}m(\pi)\left|\left|\pi(\Phi_{n,S,t,\zeta}^{\sharp})\right|\right|_{\text{HS}}^2.
\end{eqnarray}
The multiplicities of $\pi_{\infty}$ can be written as:
\begin{eqnarray}\label{26}
m(\pi_{\infty})=\sum\limits_{\Pi_{\text{cusp}}(G(\mathbb{Q}_{S}),\tau)}^{'}m(\pi^')\text{dim}\mathcal{H}_{\pi^'}^{K_{S'}^{'}},
\end{eqnarray}
where the sum is over $\pi^'$ whose Archimedean component is $\pi_{\infty}$. If we let $n \to \infty$, we have that $\pi(\Phi_{n,t,S,\zeta}^{\sharp}) \to \text{dim}\mathcal{H}_{\pi^'}^{K_{S'}^{'}}\pi_{\infty}(H_{\infty,t,\zeta}^{\sharp})$. Hence rewriting the $(25)$ we get
\begin{eqnarray}\label{27}
\sum\limits_{\lambda}(e_{\lambda}\star \phi_{n,S,t,\zeta},e_{\lambda}\star \phi_{n,S,t,\zeta})= \sum\limits_{\Pi_{\text{cusp}}(G(\mathbb{R}),\tau)}m(\pi_{\infty})\left|\left|\pi_{\infty}\Bigg(H_{\infty,t,\zeta}^{\sharp}\Bigg)\right|\right|_{\text{HS}}^2.
\end{eqnarray}
The sum on the right hand side of \eqref{27}, could be divided into two parts, $\lVert \lambda_{\pi_{\infty}}\rVert^2 \leq t^{-2}$ and $\lVert \lambda_{\pi_{\infty}}\rVert^2 \geq t^{-2}$. Here $\lambda_{\pi_{\infty}}$ denotes the infinitesimal character of $\pi_{\infty}$. The representations $\pi_{\infty}$ is a subrepresentations of a non-unitary principle series representation with parameters $\omega_{\infty}\otimes\nu_{\infty}\otimes 1$. The infinitesimal character of $\lambda_{\pi_{\infty}}$ can be written as $\nu_{\omega_{\infty}}+\nu_{\infty}$, where $\nu_{\omega_{\infty}}$ is the infinitesimal character of $\omega_{\infty}$ \cite[Prop. 8.22]{Kn}. Let $d_{\omega_{\infty}}$ be the formal degree of $\omega_{\infty}$. We have the following inequality of Hilbert-Schmidt norm in terms of Fourier transform $\widehat{H_{\infty,t,\zeta}^{\sharp}}(\omega_{\infty},\nu_{\infty})$: 
\begin{eqnarray*}
\left|\left|\widehat{H_{\infty,t,\zeta}^{\sharp}}(\omega_{\infty},\nu_{\infty})\right|\right|^2\geq d_{\omega_{\infty}}\left|\left|\pi_{\infty}\Bigg(H_{\infty,t,\zeta}^{\sharp}\Bigg)\right|\right|_{\text{HS}}^2.
\end{eqnarray*}
Using the choice of the Schwartz function the right hand side of \eqref{27} is bounded by
\begin{eqnarray}\label{28}
\sum\limits_{\lVert\nu_{\omega_{\infty}}+\nu_{\infty}\rVert \leq t^{-2}}d_{\tau}m(\pi_{\infty})\text{dim}(\text{Hom}(\mathcal{H}_{\pi_{\infty}}(\tau),V_{\tau}))\left|\psi(t\nu_{\infty})\right|^2
\end{eqnarray}
Hence, using our earlier notations, we have 
\begin{eqnarray}\nonumber
&\sum\limits_{\lVert\nu_{\omega_{\infty}}+\nu_{\infty}\rVert \leq t^{-2}}d_{\tau}m(\pi_{\infty})\text{dim}(\text{Hom}_{K_{\infty}}(\mathcal{H}_{\pi_{\infty}}(\tau),V_{\tau}))\left|\psi(t\nu_{\infty})\right|^2 &\leq d_{\tau}N_{\text{cusp}}^{\Gamma}(t^{-2},\tau)
\end{eqnarray}

\section{Main Theorem}
In this last section we put all our earlier result together to prove pur main asymptotic formula.
Suppose $\Delta_{\tau}$ is the self-adjoint Casimir operator acting on $L_{\text{cusp}}^2(\Gamma \backslash G_{S},\tau)$ with pure point spectrum $0< \nu_{0}(\tau)\leq \nu_{1}(\tau) \leq \nu_{2}(\tau)... \to \infty$ Let $\mathcal{E}(\nu_{i}(\tau))$ denote the space of eigenvectors with eigenvalue $\nu_{i}(\tau)$. Define
$$N_{\text{cusp}}^{\Gamma}(T^2,\tau)= \sum_{\nu_{i}(\tau)\leq T^{2}}\text{dim}\mathcal{E}(\nu_{i}(\tau)).$$
Let $M$ be a Riemannian Manifold. Suppose $C(M)$ denotes the product of volume of $M$, the volume of the Euclidean unit ball in $\mathbb{R}^{\text{dim}(M)}$ and $(2\pi)^{-\text{dim}(M)}$.    
Collecting all the results in the previous section, we prove the following:
\begin{theorem*}
Let $G$ be a semi-simple, connected, algebraic group over $\mathbb{Q}_{S}$. Assume that $G$ is also split and adjoint type. Let $\Gamma \subset G(\mathbb{Z}[S^{-1}])$ be a congruence subgroup with no torsion element. Let $X_{\infty}=G_{\infty}/ K_{\infty}$ and $d=\text{dim}_{\mathbb{R}}X_{\infty}$. Let $\tau$ be an irreducible representation of $K_{\infty}$ of dimension $d_{\tau}$. Then there exists a constant $C(\Gamma \backslash X_{\infty})>0$, such that
$$N_{\text{cusp}}^{\Gamma}(T^{2},\tau) \sim d_{\tau}C(\Gamma \backslash X_{\infty})T^{d} \quad \text{as} \quad T\to \infty.$$
\end{theorem*}
\begin{proof}
We make a change of variable $t=\frac{1}{T}$, and prove the asymptotic as $\limsup_{t \to 0}$.
Let us apply the partial trace formula in $(17)$ with $\phi$ being the test function $\phi_{n,S,t,\zeta}\star \widetilde{\phi_{n,S,t,\zeta}}$ in Section 8. Taking the limit as $n \to \infty$ and using $(21)$ the inequality becomes 
\begin{eqnarray}
\text{Tr}((h_{S,t,\zeta}\star \widetilde{h_{S,t,\zeta}})(e))\text{Vol}(\Omega)+\sum_{\gamma \in Z}\int_{\Omega}\text{Tr}(h_{S,t,\zeta}\star \widetilde{h_{S,t,\zeta}})(x^{-1}\gamma x)\leq N_{\text{cusp}}^{\Gamma}(t^{-2},\tau).
\end{eqnarray}
Now from Section 8 and Lemma 7 we can conclude that the term corresponding to the identity class will be asymptotic to $d_{\tau}\alpha(G)t^{-d}\text{Vol}(\Gamma \backslash G_{S})$ as $\limsup_{t \to 0}$ and as $\lim_{n\to \infty}$. And from $(24)$ section 9 we can show that as we take $\limsup_{t\to 0}$ the terms corresponding to non-identity class will converge to 0. This is done exactly as in the proof of the lower bound in the Weyl law, \cite[page 243]{LV}.  
There exist $\Gamma_{\infty,i}$, for finitely many $i$ such that 
$$\Gamma \backslash G_{S}/K_{S}= \bigcup\limits_{i} \Gamma_{\infty,i} \backslash G(\mathbb{R})/K_{\infty}.$$
For each $i$, let $N_{i}^{\Gamma}(T,\tau)$ be the eigenvalue counting function for the space $\Gamma_{\infty,i} \backslash G(\mathbb{R})/K_{\infty}$.
That this same asymptotic term along with the constant $C(\Gamma_{\infty,i} \backslash X_{\infty})$ is an upper bound for the right hand side has been proved by in greater generality by Donnelly \cite{Do}. To prove,
$$\alpha(G)\text{Vol}(\Gamma \backslash G_{S})=\sum\limits_{i}C(\Gamma_{\infty,i} \backslash X_{\infty}),$$ we argue as in \cite[Sec.6.3]{LV}. Therefore it establishes the asymptotic formula in the statement of the theorem.

\end{proof}
\subsection*{Declaration} The author declares that there is no conflict of interest. Any dataset that may have been generated during the current study is available from the author.

\end{document}